\documentclass[12pt,reqno]{amsart}
\usepackage{amsmath,amssymb,latexsym,textcomp,mathrsfs}
\usepackage[all]{xy}
\usepackage{graphicx}
\usepackage{bm,amsmath, amsthm, amssymb, amsfonts}
\usepackage{times}
\usepackage[english]{babel}

\setbox0=\hbox{$+$}
\newdimen\plusheight
\plusheight=\ht0
\def\+{\;\lower\plusheight\hbox{$+$}\;}

\setbox0=\hbox{$-$}
\newdimen\minusheight
\minusheight=\ht0
\def\-{\;\lower\minusheight\hbox{$-$}\;}

\setbox0=\hbox{$\cdots$}
\newdimen\cdotsheight
\cdotsheight=\plusheight
\def\cds{\lower\cdotsheight\hbox{$\cdots$}}

\numberwithin{equation}{section}
\theoremstyle{plain}
\newtheorem{theorem}{Theorem}[section]
\newtheorem{lemma}{Lemma}[section]
\newtheorem{example}{Example}[section]
\newtheorem{corollary}{Corollary}[section]
\newtheorem{definition}{Definition}[section]

\newtheorem{remark}{Remark}[section]
\newtheorem{note}{Note}[section]

  \newenvironment{nouppercase}{%
   \renewcommand{\uppercasenonmath}[1]{}}{}
	
	 \newcommand{\Keywords}[1]{\par\noindent
   {\small{Keywords and phrases}: #1}}
   
   \newcommand{\AMS}[1]{\par\noindent
   {\small{AMS Subject Classification (2020)}: #1}}

\begin{document}

\title{SOME ASPECTS of $s \lambda $-Closed Sets on Separation Axioms And Compactness in GT-spaces}

\author{Amar Kumar Banerjee$^1$}
\author{Jagannath Pal$^2$}
\newcommand{\acr}{\newline\indent}
 \maketitle
 \address{{1\,} Department of Mathematics, The University of Burdwan, Golapbag, East Burdwan-713104,
 West Bengal, India.
 Email:akbanerjee@math.buruniv.ac.in\acr
 {2\,} Department of Mathematics, The University of Burdwan, Golapbag, East Burdwan-713104,
 West Bengal, India. Email:jpalbu1950@gmail.com  \\}

\begin{abstract}
Here we have  investigated some  aspects of $s\lambda$-closed sets on separation axioms including $s T_{2\frac{1}{2}} $ and  $s T_{3\frac{1}{2}} $ axioms and on compactness in generalized topological spaces. 
\end{abstract}

\begin{nouppercase}
\maketitle
\end{nouppercase}

\let\thefootnote\relax\footnotetext{
\AMS{Primary 54A05, 54A10, 54D10}

}

\qquad

\section{\bf Introduction}
\label{sec:int}
A. D. Alexandroff \cite{AD} generalized   a topological space (1940) to  $ \sigma $-space (Alexandroff space) by weakening union condition where only countable union of open sets are taken to be open. Since then many topologists began to think of  various generalizations of topological spaces and the ideas of  generalized topological spaces, bitopological spaces, bispaces etc. were introduced by making variations of the criteria of topology. Plenty of works have been done in this area some of which are mentioned here for reference \cite{AA, AR, BP,AC, CA, SM}. On the other hand, many authors brought in generalization of closed sets of topological spaces to obtain various types of generalized closed sets and studied  various topological properties of the generalized closed sets on different spaces \cite{JP1, JP2, BC, DM, MBD}. The idea of generalized closed sets (g-closed sets) in a topological space was given by Levine \cite{NL}. In 1987, Bhattacharyya and Lahiri \cite{BL} introduced the class of semi-generalized closed sets in a topological space. By taking an equivalent form of $ g $-closed sets,  M. S. Sarsak (2011) \cite{MS} introduced $ g_\mu $-closed sets in a generalized topological space and investigated various related topological properties of the set with low separation axioms with the help of $ \lambda_\mu $-closed sets. 

Here we have studied  the idea of semi generalized closed sets in a more general structure in generalized topological spaces by introducing  $sg_\lambda$-closed sets with the concept of $ s\lambda $-closed sets and  investigated how far several results of topological spaces are valid with $ s\lambda $-closed sets on separation axioms and on compact spaces, regular spaces, normal spaces, completely normal spaces, completely Hausdorff spaces and completely regular spaces. Moreover, some characterizations of these spaces have been derived.

\section{\bf Preliminaries, basic definitions and some results}
 \label{sec:pre}
 
A generalized topology  \cite{MS} on a nonempty set $ Y $ is the collection $\gamma $ of subsets of $ Y $ such that $\gamma $ is closed under arbitrary unions and  $ \emptyset \in \gamma $. Members of $ \gamma $ are called $ \gamma $-open sets and complements are $ \gamma $-closed sets and the ordered pair $ (Y, \gamma) $ is known as a generalized topological space  (GTspace in short).  Definition of $ \gamma $-closure of a set $ A $ denoted by $ \overline{A_\gamma} $ is similar to that of a topological space. 
When there is no confusion, a  GTspace  $(Y , \gamma)$ will simply be denoted by $ Y $ and  sets are always subsets of $ Y $ unless otherwise stated.

\begin{definition} \label{1} \cite{BL}.
A set $ D $ of $ Y $ is said to be semi $ \gamma $-open ($ s\gamma $-open) if there exists a $ \gamma $-open set $ V $ such that $ V\subset D \subset \overline{V_\gamma} $. $ D $ is  semi $ \gamma $-closed ($ s\gamma $-closed) if  $ Y - D $ is $ s\gamma $-open.
\end{definition}

\begin{definition}\label{3} \cite {JP2}. We define semi-kernel of a set $ D $ of $ Y $ with respect to $ \gamma $ denoted by $ sker_\gamma(D)= \cap \{ G: D\subset G, G $ is $ s\gamma $-open\}. \end{definition}

\begin{definition} \label{3A}(c.f.\cite {MS}). 
A subset $D$ of $ (Y,  \gamma) $  is said to be $ s\lambda$-closed if $ D = M\cap N $  where $ M=sker_\gamma(M) $   and $ N $ is a $ s\gamma $-closed set. $ D $ is $ s\lambda$-open if $ Y -D $ is  $ s\lambda $-closed.
\end{definition}

\begin{definition}\label{4} (c.f.\cite{BN}, \cite{JRM}). Let $ (Y,\gamma) $ be a GTspace, we require the following definitons: 
 
(1) a point $ y\in Y $ is said to be a $ s\gamma $-adherence (resp. $ s\lambda $-adherence) point of a set $ D $ of $ Y $ if for every $ s\gamma $-open (resp. $ s\lambda $-open) set $ V $ containing $ y $ such that $ D\cap V\not=\emptyset $. The set of all $ s\gamma $-adherence (resp. $ s\lambda $-adherence) points of $ D $ is called $ s\gamma $-closure (resp. $ s\lambda $-closure) of $ D $ and is denoted by $ \overline{sD_\gamma} $ (resp. $ \overline{sD_\lambda}) $;
 
(2)  $ s\gamma $-interior (resp. $ s\lambda $-interior) of a set $ D $ of $ Y $ is defined as the union of all $ s\gamma $-open (resp. $ s\lambda $-open) sets contained in $ D $ and is denoted by $ sInt_\gamma(D) $ (resp. $ sInt_\lambda(D) $).
 \end{definition}
 
Clearly $ y\in \overline{sD_\gamma} $ (resp. $ y\in \overline{sD_\lambda} $) if and only if  any $ s\gamma $-open (resp. $ s\lambda $-open) set containing $ y $ intersects $ D, D \subset Y $. Again
we see that  $ \gamma $-open set $ \Rightarrow s\gamma $-open set $ \Rightarrow s\lambda $-open set and $ \gamma $-closed set $ \Rightarrow s\gamma $-closed set $ \Rightarrow s\lambda $-closed set. It can be easily verified that arbitrary intersection of $ s\gamma $-closed sets is $ s\gamma $-closed and arbitrary intersection of sets $ M $ where $ M=sker_\gamma(M) $ is a set like $ M $ and hence arbitrary intersection of $ s\lambda $-closed sets is  $ s\lambda $-closed; so arbitrary union of $ s\gamma $-open (resp. $ s\lambda $-open) sets is $ s\gamma $-open (resp. $ s\lambda $-open) and so collection of $ s\gamma $-open (resp. $ s\lambda $-open) sets forms  generalized topology on $ Y $.

\begin{theorem}\label{5}(c.f.\cite{BC}).
Let $ D,E $ be subsets of $ Y $; then for $ s\gamma $-closure  following hold:
 
(1) $  \overline{sD_\gamma}=\bigcap\{P:D\subset P; P $ is $ s\gamma $-closed\}; \quad \qquad (2) $  \overline{sD_\gamma}$  is $ s\gamma $-closed;

(3) if $ D\subset E $, then $ \overline{sD_\gamma} \subset \overline{sE_\gamma}
$;  \qquad  \qquad \qquad \qquad(4) $ D \subset\overline{sD_\gamma} $ and $\overline{s \overline{(sD_\gamma)}_\gamma}= \overline{sD_\gamma} $;

(5) $ D $ is $ s\gamma $-closed if and only if $ D=\overline{sD_\gamma} $. \quad \quad

And for $ s\gamma $-interior for the subsets $ D,E $  of $ Y $ following hold:
 
(1)  $ sInt_\gamma(D) \subset D $; \quad\qquad \qquad  \qquad \qquad (2) if $ D\subset E $, then  $ sInt_\gamma(D)  \subset  sInt_\gamma(E) $;
 
(3) $  sInt_\gamma(D) $ is $ s\gamma $-open; \qquad \qquad \qquad \quad (4) $ D $ is $ s\gamma $-open if and only if $  sInt_\gamma(D) =D$.

The results of theorem \ref{5} are also valid for $ s\lambda $-closure as well as $ s\lambda $-interior.
\end{theorem}

\begin{lemma}(c.f.\cite {JP2}).\label{2}
Suppose $ D\subset Y $; then $ sker_\gamma(sker_\gamma(D))=sker_\gamma(D)$ 
\end{lemma}

\begin{lemma}\label{6}(c.f.\cite {AJ}).
A set $ D $ of $ Y $ is $ s\lambda $-closed if and only if $
 D=sker_\gamma(D) \cap \overline{sD_\gamma} $.
\end{lemma}

\begin{note}\label{7A}
We take analogously the definitions of finite intersection property, $ s\lambda G_\delta $-sets, $ s\lambda $-weakly separated, $ s\lambda $-strongly separared, $ s\lambda $-open cover, $ s\lambda $-neighbourhood and separation axioms viz. $ s\lambda T_0,  s\lambda T_1,$     $ s\lambda T_2, s\lambda T_3, s\lambda T_4, s\lambda T_5 $ axioms parallel to those of a topological space by taking $ s\lambda $-open sets, $ s\lambda $-closed sets instead of open sets, closed sets respectively.
\end{note}

\begin{lemma}\label{8}\cite{AJ}.
Suppose $ D\subset Y $, then $ Y - \overline{s(Y -D)_\lambda} = sInt_\lambda(D) $.
\end{lemma}

\begin{lemma}\label{8A}(c.f.\cite{CG}).
Two non-empty sets $ G, H $ of $ (Y,\gamma) $ are $ s\lambda $-weakly separared if and only if $ (G\cap \overline{sH_\lambda})\bigcup (H\cap \overline{sG_\lambda})=\emptyset $.
\end{lemma}

\begin{lemma}\label{17}\cite{AJ}.
A GTspace $ (Y,\gamma) $ is $ s\lambda T_1$  if and only if every singleton of $ Y $ is $ s\lambda $-closed.
\end{lemma}

\section{\bf \quad $ sg_\lambda $-closed-sets and $ s\lambda$-Hausdorff,
$ s\lambda $-compact, $ s\lambda$-regular,  $ s\lambda $-normal GTspaces}

In this section we define $ sg_\lambda $-closed set and discuss about the effects of $ s\lambda $-closed set on $ s\lambda$-Hausdorff, $ s\lambda $-compact, $ s\lambda$-regular  and $ s\lambda $-normal GTspaces. We investigate the relation among $ s\lambda T_0,s\lambda T_1,s\lambda T_2,s\lambda R_0 $ and $ s\lambda R_1 $ axioms.

\begin{definition}\label{7} (c.f.\cite{NL}). A subset $ D $ of $ Y $ is said to be a $sg_\lambda$-closed set  if  $ \overline{sD_\lambda}\subset V $ whenever $ D\subset V $ and $ V $ is $ s\lambda $-open. $D$ is called $sg_\lambda$-open  if $ Y - D $ is $s g_\lambda$-closed.
\end{definition}

\begin{theorem}\label{9}
A subset $ D $ of $ Y $ is $ sg_\lambda $-open if and only if $ P\subset sInt_\lambda (D) $ whenever $ P\subset D $ and $ P $ is $ s\lambda $-closed.
\end{theorem}

\begin{proof}
Assume $ D $ is $ sg_\lambda $-open and $ P $ is $ s\lambda $-closed such that $ P\subset D $, so $ Y-D $ is $ sg_\lambda $-closed and $ Y-D\subset Y-P $, a $ s\lambda $-open set. Then $ \overline{s(Y -D)_\lambda}\subset (Y-P)\Rightarrow P\subset Y-\overline{s(Y -D)_\lambda}=sInt_\lambda(D) $. Conversely, let $ P\subset D,  P $ is $ s\lambda $-closed, then $ P\subset sInt_\lambda(D)=Y-\overline{s(Y -D)_\lambda}\Rightarrow \overline{s(Y -D)_\lambda}\subset Y-P\Rightarrow Y-D $ is $ sg_\lambda $-closed $ \Rightarrow  D $ is $ sg_\lambda $-open.
\end{proof}

\begin{remark}\label{10} (1):
Clearly in a GTspace $ (Y,\gamma) $,  every $ s\gamma $-closed set $ A $ as well as every set $ B $ where $ B=sker_\gamma(B) $ are $ s \lambda$-closed sets; also $ s \lambda$-closed set is $ sg_ \lambda$-closed. But converse may not be true as shown in Examples \ref{14} (i),   (ii) and  (iii) \cite{AJ}. Similarly, every $ s \gamma $-open set is $ s \lambda$-open and every  $ s \lambda$-open set is $ sg_ \lambda$-open. It also reveals from Example \ref{14} (ii) \cite{AJ}, that intersection of two $ s\lambda $-open sets may not be $ s\lambda $-open.
\end{remark}

\begin{definition}\label{12}(c.f.\cite{CG, LD}).
A GTspace $ (Y,\gamma) $ is said to be

(i)  $s\lambda R_0 $ if every $ s\lambda $-open set contains $ s\lambda $-closure of each of its singleton;

(ii)  $s\lambda R_1  $ if for $ p, q\in Y $ with $ p\not\in \overline{s\{q\}_\lambda} $, there exist  $ s\lambda $-open sets $ E, F$ such $ p\in E, q\in F, E\cap F=\emptyset $;

(iii)  $ s\lambda$-Hausdorff GTspace if for any pair of distinct points $ p, q\in Y $, there exist $ s\lambda $-open sets $ E, F $ containing $ p $ and $ q $ respectively such that $ E\cap F=\emptyset$.
\end{definition}
  
Obviously every $ s\lambda$-Hausdorff GTspace is $ s\lambda T_1$ GTspace.

\begin{theorem}\label{13}
If $ (Y, \gamma) $ is $s\lambda R_1  $ GTspace, then it is $s\lambda R_0 $.
\end{theorem}

\begin{proof}
Let $ G $ be a $ s\lambda $-open set of a $s\lambda R_1  $ GTspace  $ (Y, \gamma) $ and let $ p\in G, q\in Y-G $,  a $ s\lambda $-closed set. So $ \overline{s\{q\}_\lambda}\subset Y-G\Rightarrow p\not\in \overline{s\{q\}_\lambda} $. Since $ (Y, \gamma) $ is a $s\lambda R_1  $ GTspace, there exist $ s\lambda $-open sets $ E,F $ such that $ p\in E, q\in F; E\cap F=\emptyset $, hence $ p\in E\subset Y-F $, a $ s\lambda $-closed set. This implies that $ \overline{s\{p\}_\lambda}\subset Y-F\Rightarrow F\subset Y- \overline{s\{p\}_\lambda}\Rightarrow q\in F\subset  Y- \overline{s\{p\}_\lambda} $. Since $ q $ has been choosen arbitrarily in $ Y-G $, $ Y-G\subset Y- \overline{s\{p\}_\lambda} \Rightarrow \overline{s\{p\}_\lambda} \subset G$. This concludes $ (Y,\gamma) $ is a $ s\lambda R_0 $ GTspace.
\end{proof}

The definitions  of $ s\lambda$-regular, $ s\lambda$-normal and $ s\lambda$-completely normal GTspaces what we have taken here following \cite{SS} are somewhat different from usual definitions. We shall see in the sequel how  topological properties behave with these definitions in GTspaces. 

\begin{definition}\label{14}(c.f.\cite{SS}).
A GTspace $ (Y,\gamma) $ is said to be $ s\lambda$-regular if for any $ s\lambda $-closed set $ A $ of $ Y $ and a point $ p\in Y, p\not\in A $, there exist $ s\lambda $-open sets $  E,F $ containing $ A $ and $ \{p\} $ respectively such that $ \overline{sE_\lambda}\cap\overline{sF_\lambda}=\emptyset$.
\end{definition}

\begin{theorem}\label{15}
Let $ (Y,\gamma) $ be a GTspace, then the  following are equivalent 

(1) $ (Y,\gamma) $ is $ s\lambda$-regular.

(2) For each point $ y\in Y $ and for each $ s\lambda $-open set $ G $ containing $ y $, there exists a $ s\lambda $-open set $ E $ such that $ y\in E\subset \overline{sE_\lambda}\subset G $.

(3) For each $ s\lambda $-closed set $ P $ of $ Y, \quad\bigcap \{\overline{sF_\lambda}: P\subset F, F$ is $ s\lambda $-open\}$=P $.

(4) For each subset $ D\subset Y $ and for each $ s\lambda $-open set $ E $ such that $ D\cap E \not=\emptyset $, there exists a $ s\lambda $-open set $ F $ such that $ D\cap F \not=\emptyset $ and $ \overline{sF_\lambda}\subset E $.

(5) For each non-empty subset $ D\subset Y $ and for each $ s\lambda $-closed set $ P\subset Y $ such that $ D\cap P=\emptyset $, there exist $ s\lambda $-open sets $ M,N $ such that $ D\cap M \not=\emptyset, P\subset N $ and $ M\cap N=\emptyset $.

(6) For each $ s\lambda $-closed set $ P\subset Y $ and  $ y\in Y, y\not\in P $, there exist $ s\lambda $-open set $ E $ and $ sg_\lambda $-open set $ F $ such that $ y\in E, P\subset F $ and $ E\cap F=\emptyset $.

(7) For each $ D\subset Y $ and for each $ s\lambda $-closed set $ P\subset Y $ such that $ D\cap P=\emptyset $, there exist $ s\lambda $-open set $ E $ and $ sg_\lambda $-open set $ F $ such that $ D\cap E \not=\emptyset, P\subset F $ and $ E\cap F=\emptyset$. 

(8) For each $ s\lambda $-closed set $ P $ of $ Y $ and a point $ y\in Y, y\not\in P $, there exist $ s\lambda $-open sets $ E,F $ containing $ \{y\} $ and $ P $ respectively such that $ E\cap F=\emptyset$.
\end{theorem}

\begin{proof}
(1) $\Rightarrow $ (2): Let $ G $ be a $ s\lambda $-open set of   a $ s\lambda$-regular GTspace $ (Y,\gamma) $ and $ y\in G $. Then $ y \not\in Y-G $, a $ s\lambda $-closed set. So there exist $ s\lambda $-open sets $ E,F $ containing $ \{y\} $ and $ Y-G $ respectively such that $ \overline{sE_\lambda}\cap\overline{sF_\lambda}=\emptyset$, hence $ E\cap F=\emptyset \Rightarrow E\subset Y-F $, a $ s\lambda $-closed set and $ y\in E\subset \overline{sE_\lambda}\subset \overline{s(Y-F)_\lambda}=Y-F \subset G $.

(2) $ \Rightarrow $ (1):  Suppose the conditions hold and $ P $ is a $ s\lambda $-closed set, $ y\in Y, y\not\in P $, then $ y\in Y-P $, a $ s\lambda $-open set. By assumption, there exists a $ s\lambda $-open set $ E $ such that $ y\in E\subset \overline{sE_\lambda}\subset Y-P \Rightarrow P\subset Y- \overline{sE_\lambda}$, a $ s\lambda $-open set $ \Rightarrow E\cap (Y- \overline{sE_\lambda})=\emptyset$. Since $ y\in E $, a $ s\lambda $-open set, by assumption, there exists a $ s\lambda $-open set $ K $ such that $ y\in K\subset \overline{sK_\lambda}\subset E $ and so $ \overline{sK_\lambda}\cap (Y- \overline{sE_\lambda})=\emptyset$. If $ \overline{sK_\lambda}\cap \overline{s(Y- \overline{sE_\lambda})_\lambda}\not=\emptyset$, then there exists $ x\in  \overline{sK_\lambda}\cap \overline{s(Y- \overline{sE_\lambda})_\lambda}$. So $ x\in \overline{sK_\lambda}\subset E $ and $ x\in \overline{s(Y- \overline{sE_\lambda})_\lambda} $. But $ E $ is a $ s\lambda $-open set containing $ x $ which does not intersect $ (Y- \overline{sE_\lambda}) $ (since $ E\cap (Y- \overline{sE_\lambda})=\emptyset $). Thus we must have $ \overline{sK_\lambda}\cap \overline{s(Y- \overline{sE_\lambda})_\lambda}=\emptyset$. Hence the GTspace is $ s\lambda $-regular.

(2) $ \Rightarrow $ (3): Let $ P $ be any $ s\lambda $-closed set of $ Y $, then $ P\subset \bigcap \{\overline{sF_\lambda}: P\subset F, F$ is $ s\lambda $-open\}.

Conversely, let $ y\in Y-P $, a $ s\lambda $-open set, so $ y\not\in P $, a $ s\lambda $-closed set. Then by (2) there exists a $   s\lambda $-open set $ E $ such that $ y\in E\subset \overline{sE_\lambda}\subset Y-P \Rightarrow P\subset Y- \overline{sE_\lambda}=F$ (say), a $   s\lambda $-open set. So $ E\cap F=\emptyset $. Hence $ F\subset Y-E \Rightarrow \overline{sF_\lambda}\subset \overline{s(Y-E)_\lambda}=Y-E$. Since $ y\not\in Y-E, y\not\in \overline{sF_\lambda} $. Then $ y\not\in \bigcap \{\overline{sF_\lambda}: P\subset F, F $ is $ s\lambda $-open\}. Hence, $ P\supset \bigcap \{\overline{sF_\lambda}: P\subset F, F$ is $ s\lambda $-open\}. Hence the result follows.

(3) $\Rightarrow$ (4): Assume $ D $ and  a $ s\lambda $-open set $ E $ are subsets of $ Y $ and $ D\cap E\not=\emptyset $. So $ y\in D\cap E
\Rightarrow y\not\in Y-E $, a $ s\lambda $-closed set, hence by (3), there exists a $ s\lambda $-open set $ K $ such that $ Y-E\subset K $ and $ y\not\in \overline{sK_\lambda} $. Put  $ F= Y-\overline{sK_\lambda} $, a $ s\lambda $-open set containing $ y $ and hence $ D\cap F\not=\emptyset $. Now $ F\subset Y-K \Rightarrow \overline{sF_\lambda}\subset \overline{s(Y-K)_\lambda} =Y-K\subset E$. 

(4) $\Rightarrow$ (5): Let $ P $ be a $ s\lambda $-closed set and $ D $ be a non-empty set of $ Y $ such that $ D\cap P=\emptyset $. So $ Y-P $ is a $ s\lambda $-open set and $ D\cap (Y-P)\not=\emptyset $. Then by (4), there exists a $ s\lambda $-open set $ M $ such that $ D\cap M\not=\emptyset $ and $ \overline{sM_\lambda}\subset Y-P$. Put $ N=Y- \overline{sM_\lambda}$, then $ N $ is a $ s\lambda $-open set such that $ P\subset Y- \overline{sM_\lambda}=N $ and $ M\cap N=\emptyset $.

(5) $ \Rightarrow $ (2): Suppose the conditions hold and $ y\in G \subset Y, G $ is a $ s\lambda $-open set, then $ Y-G $ is a $ s\lambda $-closed set and $ (Y-G)\cap \{y\}=\emptyset $. By supposition, there exist $ s\lambda $-open sets $ M, N $ such that $ y\in M, Y-G\subset N $ and $ M\cap N=\emptyset $. So $ M\subset Y-N \Rightarrow \overline{sM_\lambda}\subset \overline{s(Y-N)_\lambda}=Y-N $. Hence $ y\in M\subset \overline{sM_\lambda}\subset Y-N \subset G $.

(1) $ \Rightarrow $ (6): Suppose $ P $  is a $ s\lambda $-closed set of $ Y $ and  $ y\in Y, y\not\in P $. Then by regularity criterion, there exist $ s\lambda $-open sets $ E,F $ such that $ y\in E $ and $ P\subset F $ and $ \overline{sE_\lambda}\cap \overline{sF_\lambda}=\emptyset $, so $ E\cap F=\emptyset $. Since $ s\lambda $-open set is $ sg_\lambda $-open,  the result follows.

(6) $ \Rightarrow $ (7): Suppose $ P $ is a $ s\lambda $-closed set of $ Y $ and $ y\in D\subset Y $ with $ D\cap P=\emptyset $. Then $ y\not\in P $ and hence by (6), there exist a $ s\lambda $-open set $ E $ and a $ sg_\lambda $-open set $ F $ such that $ y\in E, P\subset F $ and $ E\cap F=\emptyset $. So $ D\cap E\not=\emptyset $. 

(7) $ \Rightarrow $ (8): Let $ P $ be a $ s\lambda $-closed set of $ Y $ and $ y\in Y $ such that $ y\not\in P $. Since $ \{y\}\cap P=\emptyset $, by (7), there exist $ s\lambda $-open set $ E $ and $ sg_\lambda $-open set $ H $ such that $ y\in E, P\subset H $ and $ E\cap H=\emptyset $. Then by Theorem \ref{9}, $ P\subset sInt_\lambda(H)=F $ (say), a $ s\lambda $-open set and so $ E\cap F =\emptyset$.  

(8) $ \Rightarrow $ (2): Suppose the conditions hold and  $ G $ is a $ s\lambda $-open set of $ Y $ and $ y\in G $. Then $ y\not\in Y-G $, a $ s\lambda $-closed set. By supposition, there exist $ s\lambda $-open sets $ E,F $ such that $ y\in E, Y-G\subset F $ and $ E\cap F=\emptyset $. So $ E\subset Y-F\Rightarrow \overline{sE_\lambda}\subset \overline{s(Y-F)_\lambda}=Y-F $. Hence $ y\in E\subset \overline{sE_\lambda}\subset Y-F \subset G $. 
\end{proof}

\begin{theorem}\label{16}
Every  $ s\lambda$-regular GTspace $ (Y,\gamma) $ is $ s\lambda R_1 $.
\end{theorem}
\begin{proof}
Let $ (Y,\gamma) $ be a $ s\lambda$-regular GTspace and $ a,b\in Y $ such that $ a\not\in \overline{s\{b\}_\lambda} $. As $ \overline{s\{b\}_\lambda} $ is  $ s\lambda $-closed,   there exist $ s\lambda$-open sets $ E,F $ containing respectively $ \{a\} $ and $ \overline{s\{b\}_\lambda} $ such that $ \overline{sE_\lambda}\cap  \overline{sF_\lambda}=\emptyset$. Then $ E\cap F=\emptyset $ and so $ (Y,\gamma) $ is a $ s\lambda R_1$  GTspace by Definition \ref{12} (ii).
\end{proof}

\begin{theorem}\label{18}
A GTspace $( Y, \gamma) $ is $ s\lambda T_2 $  if and only if it is $ s\lambda R_1$ and $ s\lambda T_1 $. 
\end{theorem}
\begin{proof}
Obviously a $ s\lambda T_2 $ GTspace is  $ s\lambda T_1 $. Let $ a,b\in Y $ and $ a\not\in \overline{s\{b\}_\lambda}=\{b\} $, since the GTspace is $ s\lambda T_1 $. So $ a\not=b $. Since $ (Y,\gamma) $ is $ s\lambda T_2 $, there exist $ s\lambda $-open sets $ E,F $ such that $ a\in E, b\in F, E\cap F=\emptyset $. Hence $ (Y,\gamma) $ is also $ s\lambda R_1$ GTspace.

Conversely, let $(Y,\gamma) $ be $ s\lambda R_1$ and $ s\lambda T_1$ GTspace. So each singleton of $ Y $ is $ s\lambda $-closed. Now suppose  $ a,b\in Y, a\not= b $. Then $ a\not\in \{b\} =\overline{s\{b\}_\lambda} $ which implies that there exist $ s\lambda $-open sets $ E,F $ such that $ a\in E, b\in F, E\cap F=\emptyset $. Hence $ (Y,\gamma) $ is $ s\lambda T_2$ GTspace.
\end{proof}

We require the following definitions for $ s\lambda $-compact and $ s\lambda $-paracompact GTspaces. We assume the definitions of $ s\lambda $-Lindel$ \ddot{o} $f and $ s\lambda $-countably compact GTspaces which are parallel to that of a topological space just taking $ s\lambda $-open sets in stead of open sets.

\begin{definition}
(c.f.\cite{JRM})\label{19}. (1): A cover $ \mathcal{B} $  of $ Y $ is said to be a refinement of a cover $ \mathcal{F} $ of $ Y $ if for each element $ B $  of $ \mathcal{B} $ there is an element $ F $ of $ \mathcal{F} $ containing $ B $.

(2): A cover of $ Y $ is said to be locally finite if evey point of $ Y $ has a $ s\lambda $-neighbourhood that intersects only finitely many members of the cover.   
\end{definition}

\begin{definition}\label{20} (c.f.\cite{JRM}). (1): A GTspace
$ (Y,\gamma) $ is called $ s\lambda$-compact  if  every $ s\lambda $-open cover of $ Y $ has a finite subcover; equivalently, if  every $ s\lambda $-open cover of $ Y $ has a finite $ s\lambda $-open refinement.

(2): A GTspace $ (Y,\gamma) $ is $ s\lambda $-paracompact if every $ s\lambda $-open cover of $ Y $ has a locally finite  $ s\lambda $-open refinement that covers $ Y $.
\end{definition}

Clearly $ s\lambda$-compact GTspace is $ s\lambda$-paracompact.\

\begin{theorem}\label{21}
A GTspace $ (Y,\gamma) $ is $ s\lambda$-compact  if and only if for every collection of $ s\lambda $-closed sets $ \{P_\nu: \nu\in \Lambda, \Lambda $ is an index set\} in $ (Y,\gamma) $ possessing the finite intersection property, the intersection $ \bigcap \{P_\nu: \nu\in \Lambda\}$ of the entire collection is non-empty.

Proof is parallel to that of a topological space so is omitted.
\end{theorem}

\begin{theorem}\label{22}
Any $ s\lambda $-closed subset of $ s\lambda $-compact GTspace $ (Y,\gamma) $ is $ s\lambda $-compact.
\end{theorem}
\begin{proof}
Let $ D $ be a $ s\lambda $-closed subset of a $ s\lambda $-compact GTspace $ (Y,\gamma)$. Let $ \gamma_D=\{D\cap V: V\in \gamma\} $. Then clearly $ \gamma_D $ forms a generalized topology on $ D $. 

Now let $ \{P_\nu: \nu\in \Lambda\} $ be a family of $ s\lambda $-closed sets in $ (D,\gamma_D) $ possessing the finite intersection property.  As $ D $ is $ s\lambda $-closed in $ (Y,\gamma) $, the sets $ P_\nu $ are also $ s\lambda $-closed in $ (Y,\gamma) $. Thus $ \{P_\nu:\nu\in \Lambda\} $ is a family of $s\lambda$-closed sets in $ (Y,\gamma) $ possessing the finite intersection property. As $ (Y,\gamma) $ is $ s\lambda $-compact, it follows that $ \bigcap \{P_\nu:\nu\in \Lambda\} $ is non-empty and hence $ (D,\gamma_D) $ is $ s\lambda $-compact.
\end{proof}

\begin{theorem}\label{23}
Any $ s\lambda $-closed subset of a $ s\lambda $-paracompact (resp. $ s\lambda $-Lindel$ \ddot{o} $f, $ s\lambda $-countably compact)  GTspace is $ s\lambda $-paracompact (resp. $ s\lambda $-Lindel$ \ddot{o} $f, $ s\lambda $-countably compact).
\end{theorem}
\begin{proof}
We give proof of the first property and proofs of the rest are parallel to those of a topological space. Let $ D $ be a $ s\lambda $-closed subset of a $ s\lambda $-paracompact GTspace $ (Y,\gamma) $ and let $ \mathcal{G} $ be a $ s\lambda $-open covering of $ D $ by $ s\lambda $-open sets in $ (D,\gamma_D) $. For each $ G\in \mathcal{G} $, there exists a $ s\lambda $-open set $ G' $ of $ Y $ such that $ G'\cap D=G $. Then $ s\lambda $-open sets $ G' $ along with the $ s\lambda $-open set $ Y-D $ cover $ Y $. Let $ \mathcal{H} $ be a  locally finite $ s\lambda $-open refinement of this covering that covers $ Y $. Then the collection $ \mathcal{C}=\{H\cap D: H\in  \mathcal{H}\} $ is the required  locally finite $ s\lambda $-open refinement of $ \mathcal{G} $ that covers $ D $. Hence the result follows.
\end{proof}

\begin{theorem}\label{24}  A $ sg_\lambda $-closed subset of a $ s\lambda$-compact GTspace $ (Y,\gamma) $ is $ s\lambda$-compact.
\end{theorem}
\begin{proof}
Let $ D $ be a $ sg_\lambda $-closed subset of  $ s\lambda$-compact GTspace $ (Y,\gamma) $ and $ \mathcal{G}=\{G_\alpha\}_{\alpha\in \Lambda} $ be a $ s\lambda $-open cover of $ D $, $ \Lambda $ being an index set. So $ \bigcup_{\alpha\in \Lambda}\{ G_\alpha\} $ is a $ s\lambda $-open set containing $ D $ and hence by definition, $ \overline{sD_\lambda}\subset \bigcup_\alpha\{ G_\alpha\} $. But $ \overline{sD_\lambda}$ is a $ s\lambda $-closed set and so is a $ s\lambda $-compact set, by Theorem \ref{22}, and hence $ D\subset  \overline{sD_\lambda}\subset \bigcup^k_{i=1}(G_{\alpha_i}) $ for some  $ \{\alpha_1, \alpha_2,.....,\alpha_k\}\subset \Lambda $. Hence the result follows. 
\end{proof}

\begin{theorem}\label{25}
Any $ sg_\lambda $-closed subset of a $ s\lambda $-paracompact (resp. $ s\lambda $-Lindel$ \ddot{o} $f,  $ s\lambda $-countably compact) GTspace is $ s\lambda $-paracompact (resp. $ s\lambda $-Lindel$ \ddot{o} $f,  $ s\lambda $-countably compact). 

The proofs are parallel to that of the Theorem \ref{24}, so are omitted.  
\end{theorem}

\begin{remark}\label{25A}
We see from Example \ref{25B} that union of two $ s\lambda $-closed sets in a GTspace may not be $ s\lambda $-closed which affects the additivity of closure property of a topological space. This leads to derive the following Lemma \ref{26} which is needed to prove the next Theorem \ref{27}.
\end{remark}

\begin{example}\label{25B}
Union of two $ s\lambda $-closed sets in a GTspace may not be $ s\lambda $-closed.

Suppose $ Y=\{a,b,c,d\}, \gamma=\{\emptyset, \{a,b\}, \{b,c\}, \{a,b,c\}\} $, then $ (Y,\gamma) $ is a GTspace but not a topological space. So $ \gamma $-closed sets are: $ \{Y, \{c,d\}, \{a,d\}, \{d\}\} $; $ s\gamma $-open sets are: $ \{\emptyset, \{a,b\}, \{b,c\}, \{a,b,c\}, \{d\}, \{a,b.d\}, \{b,c,d\}\} $ (Definition \ref{1}); $ s\gamma $-closed sets are: \{Y, \{c,d\}, \{a,d\},\{d\}, \{a,b,c\}, \{c\}, \{a\}\}. Then by Lemma \ref{6} (1) $ sker_\gamma\{a\}\cap \overline{s\{a\}_\gamma}=\{a\}\Rightarrow \{a\} $ is a $ s\lambda $-closed set. Similarly $ \{c\} $ is also a $ s\lambda $-closed set. But $ sker_\gamma\{a,c\}\cap \overline{s\{a,c\}_\gamma}=\{a,b,c\}\not=\{a,c\}= \{a\}\cup \{c\} $. Hence the result follows.
\end{example}

\begin{lemma}\label{26}
Suppose $ E $ and $ F $ are subsets of a GTspace $(Y,\gamma)$, then $ \overline{sE_\lambda}\bigcup \overline{sF_\lambda}=\overline{s(E\bigcup F)}_\lambda$ if  union of finite number of $ s\lambda $-closure of sets is $ s\lambda $-closed. 
\end{lemma}

\begin{proof}
Let $ E,F\subset Y $. Then $ E\subset E\cup F $, $ F\subset E\cup F $ and $ \overline{sE_\lambda}\subset \overline{s(E\cup F)_\lambda} $,  $ \overline{sF_\lambda}\subset \overline{s(E\cup F)_\lambda} $, so $ \overline{sE_\lambda}\cup \overline{sF_\lambda}\subset \overline{s(E\cup F)_\lambda} $. Again by assumption, $ \overline{sE_\lambda}\cup \overline{sF_\lambda} $ is a $ s\lambda $-closed set and it contains $ E\cup F $, so $ \overline{s(E\cup F)_\lambda}\subset \overline{sE_\lambda}\cup \overline{sF_\lambda} $. Hence the result follows.
\end{proof}

\begin{theorem}\label{27}
Each $ s\lambda $-compact subset of a $ s\lambda $-regular GTspace $ (Y,\gamma) $ is $ sg_\lambda $-closed, if union of finite number of $ s\lambda $-closure of sets is $ s\lambda $-closed.
\end{theorem}
\begin{proof}
Suppose $ D $ is a $ s\lambda $-compact  subset of a $ s\lambda $-regular GTspace $ (Y,\gamma) $ and $ D\subset U, U $ being a $ s\lambda $-open set. Let $ y\in D $, so $ y\in U $. As $ (Y,\gamma) $ is a  $ s\lambda $-regular GTspace, there exists a $ s\lambda $-open set $ V_y $ such that $ y\in V_y\subset \overline{s(V_y)_{\lambda}}\subset U $. Taking union we get $\bigcup  \{y\}\subset \bigcup_{y\in D} V_y\subset \bigcup_{y \in D}\overline{s(V_y)_\lambda}\subset U $ which implies $ D\subset \bigcup_{y\in D} V_y\subset \bigcup\{\overline{s(V_y )_\lambda}:y\in D\}\subset U  $. So $\{V_{y, y\in D}\} $ is a $ s\lambda $-open cover of $ D $. For $ s\lambda $-compactness it has a finite subcover, say $ V_i, i=1.2.....n.$ and let $ V'=\bigcup V_i,{i=1,2,.....,n} $. Then $ D\subset V'\subset \overline{sV'_\lambda}\subset U $, since $ \bigcup \overline{s(V_i)_\lambda}=  \overline{s(\cup V_i)_\lambda}, i=1,2,...,n$, by Lemma \ref{26}. So $ \overline{sD_\lambda}\subset U $ and hence $ D $ is $ sg_\lambda $-closed, by definition.
\end{proof}

\begin{definition}\label{28}(c.f.\cite{SS}).
Let $ (Y,\gamma) $ be a GTspace, then it is said to be $ s\lambda$-normal  if for any pair of disjoint $ s\lambda $-closed sets $ A,B $ of $ Y $ there exist $ s\lambda $-open sets $ E,F $ containing $ A,B $ respectively such that $ \overline{sE_\lambda}\cap\overline{sF_\lambda}=\emptyset$.
\end{definition}

\begin{theorem}\label{29}
Suppose $ (Y,\gamma) $ is a GTspace, then  following are equivalent.

(1) $ (Y,\gamma) $ is $ s\lambda $-normal.

(2) For every pair of $ s\lambda $-open sets $ E,F $ whose union is $ Y $, there exist $ s\lambda $-closed sets $ A,B $ such that $ A\subset E, B\subset F, A\cup B=Y $.

(3) For every $ s\lambda $-closed set $ A $ and every $ s\lambda $-open set $ G $ containing $ A $, there exists a $ s\lambda $-open set $ E $ such that $ A\subset E\subset \overline{sE_\lambda}\subset G $.

(4) Every pair of disjoint $ s\lambda $-closed sets are $ s\lambda $-strongly separated (Note \ref{7A}).
\end{theorem}
\begin{proof}
(1) $ \Rightarrow $ (2): Let $ E,F $ be a pair of $ s\lambda $-open sets such that $ E\cup F=Y $. Then $ Y-E $ and $Y-F$ are disjoint $ s\lambda $-closed sets. Since $ (Y,\gamma) $ is a $ s\lambda $-normal GTspace, there exist $ s\lambda $-open sets $ M,N $ such that $ Y-E\subset M $ and $ Y-F\subset N $ and $ \overline{sM_\lambda}\cap \overline{sN_\lambda}=\emptyset$. Therefore, $ M\cap N=\emptyset $. Put $ A=Y-M $ and $ B=Y-N $, then $ A,B $ are $ s\lambda $-closed sets such that $ A\subset E, B\subset F, A\cup B=Y $.

(2) $ \Rightarrow $ (3): Let $ A $ be a $ s\lambda $-closed set and $ G $ be a $ s\lambda $-open set containing $ A $. Then $ Y-A $ and $ G $ are $ s\lambda $-open sets whose union is $ Y $. Then by (2), there exist $ s\lambda $-closed sets $ P,Q $ such that $ P\subset Y-A $ and $ Q\subset G $ and $ P\cup Q=Y $. Therefore $ A\subset Y-P, Y-G\subset Y-Q $ and $ (Y-P)\cap (Y-Q)=\emptyset $. Put $ E=Y-P $ and $ F=Y-Q $, then $ E,F $ are disjoint $ s\lambda $-open sets such that $ A\subset E=Y-P\subset Q= Y-F\subset G $. As $ Y-F $ is a  $ s\lambda $-closed set, we have $ \overline{sE_\lambda}\subset Y-F $ and $ A\subset E\subset \overline{sE_\lambda}\subset G $.

(3) $ \Rightarrow $ (1): Let $ P,Q $ be two disjoint  $ s\lambda $-closed sets. Then $ P\subset Y-Q $, a  $ s\lambda $-open set, hence by (3), there exists a $ s\lambda $-open set $ E $ such that $ P\subset E\subset \overline{sE_\lambda}\subset Y-Q $. Put $ Y- \overline{sE_\lambda}=F$, a $ s\lambda $-open set containing the  $ s\lambda $-closed set $ Q $. Then again by (3), there exists a $ s\lambda $-open set $ K $ such that $ Q\subset K\subset \overline{sK_\lambda}\subset F=Y- \overline{sE_\lambda}$. Therefore $ P,Q $ are contained in $ s\lambda $-open sets $ E, K $ respectively and $ \overline{sE_\lambda}\cap \overline{sK_\lambda}=\emptyset $. Hence $ (Y,\gamma) $ is a $ s\lambda $-normal GTspace.

(1) $ \Rightarrow $ (4): 
Let $ P,Q $ be a pair of disjoint  $ s\lambda $-closed sets of $ Y $. For $ s\lambda $-normality criterion, there exist $ s\lambda $-open sets $ E,F $ containing $ P,Q $ respectively such that $ \overline{sE_\lambda}\cap\overline{sF_\lambda}=\emptyset$. This implies $ E\cap F=\emptyset $ and the result follows.

(4) $ \Rightarrow $ (3): Assume $ A $ is a $ s\lambda $-closed set contained in a $ s\lambda $-open set  $ G $, then $ A $ and $ Y-G $ are a pair of disjoint $ s\lambda $-closed sets. So by the $ s\lambda $-strongly separatedness, there exist $ s\lambda $-open sets $ E,F $ containing $ A $ and $Y-G $ respectively such that $ E\cap F=\emptyset$. This implies $ E\subset Y-F\Rightarrow \overline{sE_\lambda}\subset \overline{s(Y-F)_\lambda}=Y-F\subset G $. Thus $ A\subset E\subset \overline{sE_\lambda}\subset G $. 
\end{proof}

\begin{theorem}\label{30}
Let $ (Y,\gamma) $ be a GTspace, then  following are equivalent

(1) $ (Y,\gamma) $ is $ s\lambda $-normal.

(2) For each $ s\lambda $-closed set $ P $ and each $ s\lambda $-open set $ G $ containing $ P $, there exists a $ sg_\lambda $-open set $ E $ such that $ P\subset E\subset \overline{sE_\lambda}\subset G $.

(3) For each $ s\lambda $-closed set $ P $ and each $ sg_\lambda $-open set $ G $ containing $ P $, there exists a $ sg_\lambda $-open set $ E $ such that $ P\subset E\subset \overline{sE_\lambda}\subset sInt_\lambda(G) $.

(4)  For each $ sg_\lambda $-closed set $ P $ and each $ s\lambda $-open set $ G $ containing $ P $, there exists a $ s\lambda $-open set $ E $ such that $ \overline{sP_\lambda}\subset E\subset \overline{sE_\lambda}\subset G $. 

 (5)
For each $ s\lambda $-closed set $ P $ and each $ sg_\lambda $-open set $ G $ containing $ P $, there exists a $ s\lambda $-open set $ E $ such that $ P\subset E\subset \overline{sE_\lambda}\subset sInt_\lambda(G) $.

(6) For each $ sg_\lambda $-closed set $ P $ and each $ s\lambda $-open set $ G $ containing $ P $, there exists a $ sg_\lambda $-open set $ E $ such that $ \overline{sP_\lambda}\subset E\subset \overline{sE_\lambda}\subset G $. \end{theorem}
\begin{proof}
(1) $\Rightarrow $ (2): Suppose $ P $ is a $ s\lambda $-closed set and $ G $ is a $ s\lambda $-open set containing $ P $, then $ P $   and  $ Y-G $ are two disjoint $ s\lambda $-closed sets. So, by normality criterion, there exist  $ s\lambda $-open sets $ E $ and $F $ containing $ P $ and $ Y-G $ respectively such that $ \overline{sE_\lambda}\cap\overline{sF_\lambda}=\emptyset\Rightarrow E\cap F=\emptyset $. Since $ F $ is also a $ sg_\lambda $-open set, $ Y-G\subset sInt_\lambda(F) $, by Theorem \ref{9}. By Lemma \ref{8}, we have, $ \overline{s(Y-F)_\lambda}=Y- sInt_\lambda(F)\subset G$.    Thus $ P\subset E\subset\overline{sE_\lambda}\subset G $, since $  E\subset Y-F$.

(2) $ \Rightarrow $ (1):  Suppose the conditions hold and $ P, Q $ are two disjoint $ s\lambda $-closed sets. Then $ P\subset Y-Q $, a  $ s\lambda $-open set, hence by assumption, there exists a $ sg_\lambda $-open set $ E $ such that $ P\subset E\subset \overline{sE_\lambda}\subset Y-Q \Rightarrow Q\subset Y-\overline{sE_\lambda}$, a $ s\lambda $-open set.  Then again by (2), there exists a $ sg_\lambda $-open set $ K $ such that $ Q\subset K\subset \overline{sK_\lambda}\subset Y- \overline{sE_\lambda}$. So, by Theorem \ref{9}, $ P\subset sInt_\lambda(E) $  and $ Q\subset sInt_\lambda(K) $ and  $ sInt_\lambda(E),  sInt_\lambda(K)$ are $ s\lambda $-open sets. Now $ sInt_\lambda(E)\subset E \Rightarrow \overline{sInt_\lambda(E)}\subset \overline{sE_\lambda} $ and $ sInt_\lambda(K)\subset K \Rightarrow \overline{sInt_\lambda(K)}\subset \overline{sK_\lambda} \subset Y-\overline{sE_\lambda} $. Since $ \overline{sE_\lambda} \bigcap (Y- \overline{sE_\lambda})=\emptyset,  \overline{sInt_\lambda(E)}\cap \overline{ sInt_\lambda(K)}=\emptyset $. Hence $ (Y,\gamma) $ is a $ s\lambda $-normal GTspace.

(2) $\Rightarrow $ (3): Suppose $ P $ is a $ s\lambda $-closed set and $ G $ is a  $ sg_\lambda $-open set containing $ P $. So by Theorem \ref{9}, $ P\subset sInt_\lambda(G) $, a $ s\lambda $-open set. Then by (2), there exists a $ sg_\lambda $-open set $ E $ such that $ P\subset E\subset \overline{sE_\lambda}\subset sInt_\lambda(G) $.

(3) $\Rightarrow $ (4): Suppose $ P $ is a $ sg_\lambda $-closed set and $ G $ is a  $ s\lambda $-open set containing $ P $, then $ \overline{sP_\lambda} \subset G$. Since $ G $ is also $ sg_\lambda $-open and $ \overline{sP_\lambda}$ is $ s\lambda $-closed, by (3), there exists a $ sg_\lambda $-open set $ F $ such that $ \overline{sP_\lambda}\subset F\subset\overline{sF_\lambda}\subset sInt_\lambda(G)= G $. Since $ F $ is $ sg_\lambda $-open and $\overline{sP_\lambda} $ is $ s\lambda $-closed,  we have, by Theorem \ref{9}, $\overline{sP_\lambda}\subset sInt_\lambda(F) $. Put $ sInt_\lambda(F)=E $, then $ E $ is $ s\lambda $-open and hence $ \overline{sP_\lambda}\subset E\subset \overline{sE_\lambda}=\overline{s(sInt_\lambda(F))_\lambda}\subset \overline{sF_\lambda}\subset G $, since $ sInt_\lambda (F) \subset F $.

(4) $\Rightarrow $ (5): Suppose $ P $ is a $ s\lambda $-closed set and $ G $ is a  $ sg_\lambda $-open set containing $ P $. So, by Theorem \ref{9}, $ P\subset sInt_\lambda(G) $, a $ s\lambda $-open set.  Since $ P $ is  $ sg_\lambda $-closed and $ sInt_\lambda(G) $ is  $ s\lambda $-open, then, by (4), there exists a $ s\lambda $-open set $ E $ such that $ P= \overline{sP_\lambda}\subset E\subset \overline{E_\lambda}\subset  sInt_\lambda(G)$.

(5) $\Rightarrow $ (6): Suppose $ P $ is a $ sg_\lambda $-closed set and $ G $ is a  $ s\lambda $-open set containing $ P $. Then $ \overline{sP_\lambda}\subset G $. Now $ \overline{sP_\lambda}$ is a $ s\lambda $-closed set and $ G $ is a $ sg_\lambda $-open set. So, by (5), there exists a $ s\lambda$-open set $ E $ such that $ \overline{sP_\lambda}\subset E\subset \overline{sE_\lambda}\subset sInt_\lambda(G)=G $.  

(6) $\Rightarrow $ (1): It will suffice to prove (6) $\Rightarrow $ (2) and it is obvious.
\end{proof} 

\begin{lemma}\label{31}
Let $ \mathcal{E} $ be a locally finite collection of subsets of $ (Y,\gamma) $ satisfying the condition (A): intersection of finite number of $ s\lambda $-open sets is $ s\lambda $-open, 
then $ \overline{s(\bigcup_{E\in\mathcal{E}}E)_\lambda}=\bigcup_{E\in \mathcal{E}}\overline{sE_\lambda} $.
\end{lemma}
\begin{proof}
Let $ K=\bigcup_{E\in \mathcal{E}}E $. It can be easily proved that $ \bigcup_{E\in \mathcal{E}}\overline{sE_\lambda} \subset \overline{sK_\lambda} $.

Conversely, let $ y\in \overline{sK_\lambda} $ and $ V $ be a $ s\lambda $-neighbourhood of $ y $ that intersects only finitely many elements of $ \mathcal{E} $, say, $ E_1,.....,E_n$. If  $ y $ belongs to one of the sets $ \overline{s(E_1)_\lambda},........., \overline{s(E_n)_\lambda} $, then $ y\in \bigcup_{E\in \mathcal{E}}\overline{sE_\lambda} $. Otherwise, the set $ V-\overline{s(E_1)_\lambda}- ........- \overline{s(E_n)_\lambda}=V-\bigcup^n_{i=1}\overline{s(E_i)_\lambda}=V\bigcap(Y-\bigcup^n_{i=1}\overline{s(E_i)_\lambda})$ would be a $ s\lambda $-neighbourhood of $ y $ that intersects no element of $ \mathcal{E} $ and hence does not intersect $ K $ which is a contradiction to the supposition that $ y\in \overline{sK_\lambda} $. Hence the result follows.
\end{proof}

\begin{theorem}\label{32}
Every $ s\lambda $-paracompact and $ s\lambda $-Hausdorff GTspace $ (Y,\gamma) $ satisfying the condition (A) mentioned in Lemma \ref{31} is $ s\lambda $-normal.
\end{theorem}

\begin{proof}
Let $ (Y, \gamma) $ be a $ s\lambda $-paracompact and $ s\lambda $-Hausdorff GTspace. First we will prove that $ (Y, \gamma) $ is a $ s\lambda $-regular GTspace. Let $ y\in Y $ and $ P $ be a $ s\lambda $-closed subset of $ Y $ such that $ y\not\in P $. Now take any point $ p\in P $, then by $ s\lambda $-Hausdorff condition, there are $ s\lambda $-open sets $ U_y $ and $ U_p $ containing $ y $ and $ p $ respectively such that $ U_y\cap U_p=\emptyset $. So, we can say that $ y $ is not a $ s\lambda $-adherence point of $ U_p $. Hence $ \overline{s(U_p)_\lambda}\cap \{y\}=\emptyset $. So the collection of $ \{U_p:p\in P\} $ together with $ s\lambda $-open set $ Y-P $ forms a $ s\lambda $-oepn cover for $ Y $. Since $ Y $ is $ s\lambda $-paracompact, there exists a locally finite $ s\lambda $-open refinement $ \mathcal{A} $ that covers $ Y $. Now condsider the subcollection $ \mathcal{B} $ of $ \mathcal{A} $ consisting of all those element of $ \mathcal{A} $ that intersects $ P $. Then $ \mathcal{B} $ covers $ P $. Again, if $ B\in  \mathcal{B} $, then $ \overline{s B_\lambda}\cap \{y\}=\emptyset $. Let $ V=\bigcup_{B\in  \mathcal{B}}B $; then $ V $ is a $ s\lambda $-open set in $ Y $ containing $ P $. By Lemma \ref{31}, we have $ \overline{sV_\lambda}=\overline{s(\bigcup_{B\in  \mathcal{B}}B)_\lambda}=\bigcup_{B\in \mathcal{B}}\overline{sB_\lambda} $ which implies that $ \overline{sV_\lambda}\cap \{y\}=\emptyset $. So $ s\lambda $-closure of $ V $ is disjoint from $ y $ i.e. there exists $ s\lambda $-open set $ U_y $ containing $ y $ such that $ U_y\cap V=\emptyset $. Thus by Theorem \ref{15} (8), the GTspace $ (Y,\gamma) $ is $ s\lambda $-regular. 

Now to prove $ s\lambda $-normality, let $ P $ and $ Q $ be a pair of disjoint $ s\lambda $-closed sets and $ p\in P $. Then by $ s\lambda $-regularity criterion, there is a $ s\lambda $-open set $ U_p $ containing $ p $  whose $ s\lambda $-closure is disjoint from $ Q $ i.e. $ \overline{s(U_p)_\lambda}\cap Q=\emptyset $. So for each $ p\in P $ as $ p $ runs over $ P $ we have, by adopting similar process as above, a $ s\lambda $-open set $ V $ in $ Y $ containing $ P $ such that $ \overline{sV_\lambda}\cap Q=\emptyset $. On the other hand, $ P\subset Y-Q $, a $ s\lambda $-open set and we get a $ s\lambda $-open set $ V $ such that $ P\subset V\subset \overline{sV_\lambda}\subset Y-Q $, then by Theorem \ref{29} (3), the space $ (Y,\gamma) $ is $ s\lambda $-normal.
\end{proof}

\section{\bf $ s\lambda T_{2\frac{1}{2}} $ GTspace, $ s\lambda T_{3\frac{1}{2}} $ GTspace and $ s\lambda $-completely normal GTspace}

In this section we investigate some properties of $ s\lambda T_{2\frac{1}{2}} $ axiom and $ s\lambda T_{3\frac{1}{2}} $ axiom with the help of sets separated by real valued $ s\lambda $-continuous function.

\begin{definition}\label{33}(c.f.\cite{AA}).
A function $ f: (X,\gamma_1)  \longrightarrow (Y,\gamma_2) $ is said to be $ s\lambda $-continuous if  inverse image of each $ \gamma_2 $-open set $ V $ is $ s\lambda $-open in $(X,\gamma_1)  $.

Clearly, a function $ f: (X,\gamma_1)  \longrightarrow (Y,\gamma_2) $ is said to be $ s\lambda $-continuous if  inverse image of each $ \gamma_2 $-closed set $ P $ is $ s\lambda $-closed in $(Y,\gamma_1)  $.
\end{definition}

\begin{definition}\label{34}(c.f.\cite{CG}). 
Two sets $ E,F $ of $ (Y,\gamma) $ are said to be separated by a $ s\lambda $-continuous real valued function if there exists a $ s\lambda $-continuous real valued function $ f $ on $ Y $ such that $ f(y)=0 $ for all $ y\in E $, $ f(y)=1 $ for all $ y\in F $ and $ 0\leq f(y)\leq 1 $ for all $ y\in Y $.
\end{definition}

\begin{theorem}\label{35} 
If two sets $ A,B $ in a GTspace $ (Y, \gamma) $ are separated by a $ s\lambda $-continuous real valued function $ f $ on $ Y $ they are separated by disjoint $ s\lambda $-closed neighbourhoods in $ (Y,\gamma) $.
\end{theorem}
\begin{proof}
Suppose the conditions hold, then $ f(y)=0 $ for all $ y\in A $, $ f(y)=1 $ for all $ y\in B $ and $ 0\leq f(y)\leq 1 $ for all $ y\in Y $. Let $ I $ denote the closed interval $ [0, 1] $ and $ \sigma_I $ denote the relative topology on $ I $ of usual topology $\sigma$ of the real number space $ (R,\sigma) $. Then clearly $ \sigma_I $ is also a generalized topology on $ I $. Now half open intervals $ [0, \frac{1}{4}) $ and $ (\frac{3}{4}, 1]$ are two disjoint open sets in $(I, \sigma_I)$ as well as two disjoint open sets in $ \sigma $. Since $ f $ is $ s\lambda $-continuous on $ (Y,\gamma) $ with values in $ I $, then $ f^{-1}[0, \frac{1}{4}) $  and $ f^{-1}(\frac{3}{4}, 1] $ are two disjoint $ s\lambda $-open sets in $ (Y,\gamma) $; also $ A\subset f^{-1}[0, \frac{1}{4}) $  and $ B\subset f^{-1}(\frac{3}{4}, 1] $. Again, $ [0, \frac{1}{4}] $ and $ [\frac{3}{4}, 1]$ are two disjoint closed sets in $(I, \sigma_I)$ as well as two disjoint closed sets in $ \sigma $. Since $ f $ is $ s\lambda $-continuous, then $ f^{-1}[0, \frac{1}{4}] $ and $  f^{-1}[\frac{3}{4}, 1] $ are two disjoint $ s\lambda $-closed sets in $ (Y,\gamma) $ containing the $ s\lambda $-open sets $ f^{-1}[0, \frac{1}{4}) $  and $ f^{-1}(\frac{3}{4}, 1] $ respectively. Hence the result follows. 
\end{proof}

\begin{theorem}\label{36A}
A GTspace $ (Y,\gamma) $ is $ s\lambda $-normal if every pair of disjoint $ s\lambda $-closed sets $ A,B $ of $ Y $ are separated by a $ s\lambda $-continuous real valued function $ f $ on $ Y $.
\end{theorem}

\begin{proof}
Suppose $ (Y,\gamma) $ is a GTspace satisfying the given conditions. Then by Theorem \ref{35}, $ s\lambda $-closed sets $ A,B\subset Y $ are separated by disjoint $ s\lambda $-closed neighbourhoods and so the GTspace is $ s\lambda $-normal.
\end{proof}

\begin{theorem}\label{36}(c.f.\cite{CG}).  Every pair of disjoint $ s\lambda $-closed sets $ A,B $ of $ Y $ are separated by a $ s\lambda $-continuous real valued function $ f $ on $ Y $ if  $ (Y,\gamma) $ is a $ s\lambda $-normal GTspace satisfying the condition (A) mentioned in Lemma \ref{31}.
\end{theorem}

\begin{proof}
Let  $ A, B $ be a pair of disjoint $ s\lambda $-closed sets of a $ s\lambda $-normal GTsapce $ (Y,\gamma) $ satisfying the condition (A) mentioned in Lemma \ref{31}. Then $ A\subset  Y-B $, a $ s\lambda $-open set and so by Theorem \ref{29} (3), there exists a $ s\lambda $-open set denoted by $ V(\frac{1}{2}) $ such that $ A\subset V(\frac{1}{2})$ and $ \overline{s(V(\frac{1}{2}))_\lambda}\subset Y-B $. 
Since $ A $ and $ \overline{s(V(\frac{1}{2}))_\lambda} $ are $ s\lambda $-closed sets contained in  $ s\lambda $-open sets $ V(\frac{1}{2}) $ and $ Y-B $ respectively then again by Theorem \ref{29} (3), there exist $ s\lambda $-open sets $ V(\frac{1}{4}) $ and $ V(\frac{3}{4})$ such that 
$ A\subset V(\frac{1}{4})$ and $ \overline{s(V(\frac{1}{4}))_\lambda}\subset V(\frac{1}{2})$ and $ \overline{s(V(\frac{1}{2}))_\lambda}\subset V(\frac{3}{4}) $ and $ \overline{s(V(\frac{3}{4}))_\lambda}\subset Y-B $. Repeating the process we have $ s\lambda $-open sets viz. $ V(\frac{r}{2^n}) $ for a fixed $ n $ and $ r=1,2,......., 2^n-1 $ such that $ A\subset V(\frac{1}{2^n}) $ and $ \overline{s(V(\frac{r}{2^n}))_\lambda}\subset V(\frac{r+1}{2^n}) $ where $ V(1)=Y-B $.  So, for each value of $ r $ there exists a $ s\lambda $-open set $ V(\frac{2r+1}{2^{n+1}}) $ such that $ \overline{s(V(\frac{r}{2^n}))_\lambda}\subset V(\frac{2r+1}{2^{n+1}}) $ and $ \overline{s(V(\frac{2r+1}{2^{n+1}}))_\lambda}\subset V(\frac{r+1}{2^n}) $. Thus, for every dyadic proper fraction $ \frac{r}{2^n} $ there exists a $ s\lambda $-open set $ V(\frac{r}{2^n}) $ with the properties that

(a): $ A $ is contained in each $ V(\frac{r}{2^n}) $; 
\qquad   
(b): each $ \overline{s(V(\frac{r}{2^n}))_\lambda} $ is contained in $ Y-B $ and 

(c): if $ \frac{r}{2^n}<  \frac{r'}{2^m} $ then $ \overline{s(V(\frac{r}{2^n}))_\lambda}\subset  V(\frac{r'}{2^m}) $.

Now we define a function $ f $ on $ Y $ with values in $ I=[0,1] $ as:

$ f(y)= $ the greatest lower bound of all those numbers  $ \frac{r}{2^n}$ if $ y\in V(\frac{r}{2^n}) $ and 

$ f(y)=1 $ if $ y $ does not belong to any $ V(\frac{r}{2^n}) $. Then
 
(i) for any point $ y $ in $ Y, 0\leq f(y)\leq 1 $ since $ 0< \frac{r}{2^n}<1$;

(ii) for any point $ y\in A, y $ belongs to every $V(\frac{r}{2^n}) $, so $ f(y) =0 $; and 

(iii) for any point $ y\in B, y $ does not belong to any $ V(\frac{r}{2^n}) $, so $ f(y)=1 $.

Next we shall prove the $ s\lambda $-continuity of $ f $. Let $ q $ be any point in $ Y $ and $ \epsilon $ be a pre-assigned positive number. Then following  cases may be considered:

Case-1:  Let $ f(q)=0 $ and $ m $ be a positive integer such that $ \frac{1}{2^m}<\epsilon $ and let $ H=V(\frac{1}{2^m}) $. Then $ H $ is a $ s\lambda $-open set containing $ q $. Also for any point $ y\in H, f(y)\leq \frac{1}{2^m}<\epsilon $. Since $ 0\leq f(y) $, it follows that $ \mid f(y)-f(q)\mid =f(y)< \epsilon $. So $ f $ is $ s\lambda $-continuous at $ q $.

Case-2 : Let $ f(q) =1 $ and $ m $ be a positive integer such that $ \frac{1}{2^m}<\epsilon $ and let $ H=Y-\overline{s(V(\frac{2^m-1}{2^m}))_\lambda} $. Then $ H $ is a $ s\lambda $-open set containing $ q $, for, $ q\in \overline{s(V(\frac{2^m-1}{2^m}))_\lambda} \subset V(\frac{2^{m+1}-1}{2^{m+1}})$ then $ f(q) \leq \frac{2^{m+1}-1}{2^{m+1}}<1 $ which contradicts to the assumption. Now for any point $ y\in 
H $, we must have $ y\not\in \overline{s(V(\frac{2^m-1}{2^m}))_\lambda} $ so that $ y\not\in V(\frac{2^m-1}{2^m}) $. Hence $ y\not\in V(\frac{r}{2^n}) $ whenever $ \frac{r}{2^n}< \frac{2^m-1}{2^m} $. So, $ f(y) \geq \frac{2^m-1}{2^m}=1-\frac{1}{2^m}> 1-\epsilon $. But $ f(y)\leq 1 $. Hence $ \mid f(q) -f(y)\mid =1-f(y)< \epsilon$ and so  $ f $ is $ s\lambda $-continuous at $ q $.

Case-3 : Let $ 0< f(q) <1 $ and $ n $ be a positive integer such that $ \frac{1}{2^{n-1}}<\epsilon $ and $ \frac{r}{2^n}<f(q)<\frac{r+1}{2^n}<1 $ for a suitable positive integer $ r $. Then $ q\not\in V(\frac{r}{2^n}) $ and therefore $ V\not\in V(\frac{r-1}{2^n}) $ but $ q\in V(\frac{r+1}{2^n}) $. Now let $ H=V(\frac{r+1}{2^n})-\overline{s(V(\frac{r-1}{2^n}))_\lambda}= V(\frac{r+1}{2^n}) \cap (Y-\overline{s(V(\frac{r-1}{2^n}))_\lambda}) $, then by condition (A), $ H $ is a $ s\lambda $-open set containing $ q $, since $ \overline{s(V(\frac{r-1}{2^n}))_\lambda}\subset V(\frac{r+1}{2^n}) $ and $ q\not\in V(\frac{r}{2^n}) $ and $ \overline{s(V(\frac{r-1}{2^n}))_\lambda}\subset V(\frac{r}{2^n}) $. Also for any point $ y\in H, y\in V(\frac{r+1}{2^n}) $ and $y\not\in \overline{s(V(\frac{r-1}{2^n}))_\lambda} $ and so $ y\not\in V(\frac{r-1}{2^n}) $. Therefore from the definition of the function $ f $ it reveals that 
$ \frac{r-1}{2^n}\leq f(y)\leq \frac{r+1}{2^n} $, so $ \mid f(y)-f(q)\mid \leq \frac{1}{2^{n-1}}< \epsilon $. Hence $ f $ is $ s\lambda $-continuous at $ q $.

Thus $ f $ is $ s\lambda $-continuous at every point $ q $ in $ Y $ and so $ s\lambda $-continuous on $ Y $. Hence $ A, B $ are separated by a $ s\lambda $-continuous real valued function $ f $ on $ Y $. 
\end{proof}

\begin{definition}\label{37}(c.f.\cite{SS}).
A GTspace $ (Y,\gamma) $ is said to satisfy $ s\lambda T_{2\frac{1}{2}} $ axiom if for every pair of distinct points $ p, q\in Y $, there exist $ s\lambda $-open sets $ E,F $ containing $ p, q $ respectively such that $ \overline{sE_\lambda}\cap \overline{sF_\lambda}=\emptyset $.
The GTspace which satisfies $ s\lambda T_{2\frac{1}{2}} $ axiom is called $ s\lambda $-Urysohn.
\end{definition}

\begin{definition}\label{38}(c.f. \cite{CG}).
A GTspace $ (Y,\gamma) $ is said to be $ s\lambda $-completely Hausdorff if for every pair of distinct points $ p, q\in Y $, there exists a $ s\lambda $-continuous real valued function $ f $ on $ Y $ such that $ f(p)=0, f(q)=1 $ and $ 0\leq f(y)\leq 1 $ for all $ y\in Y $ i.e.  every pair of distinct points in $ Y $ can be separated by a $ s\lambda $-continuous real valued function on $ Y $.
\end{definition}

\begin{theorem}\label{39} 

 (i):
Every $ s\lambda $-completely Hausdorff GTspace $ (Y,\gamma) $ is $ s\lambda $-Urysohn.
 
(ii) Every $ s\lambda $-Urysohn GTspace$ (Y,\gamma) $ is $ s\lambda $-Hausdorff.

(iii)  $ s\lambda $-Urysohn and $ s\lambda $-normal GTspace $ (Y,\gamma) $  satisfying the condition (A) as mentioned in Lemma \ref{31}  is $ s\lambda $-completely Hausdorff.
\end{theorem}

\begin{proof}
(i): Proof is similar to that of the Theorem \ref{35} and so is omitted. 

(ii): It is obvious from Definition.

(iii): Let the condition hold and $ p,q\in Y, p\not=q $. Since the  GTspace is $ s\lambda $-Urysohn, there exist $ s\lambda $-open sets $ E,F $ such that $ p\in E $ and $ q\in F $ and $ \overline{sE_\lambda}\cap \overline{sF_\lambda}=\emptyset $. Since the GTspace is also $ s\lambda $-normal, by Theorem \ref{36}, there exists a $ s\lambda $-continuous real valued function $ f $ on $ Y $ such that   $ f(y)=0 $ if $ y\in \overline{sE_\lambda} $ and  $ f(y) =1 $ if $ y\in \overline{sF_\lambda} $ and $ 0\leq f(y) \leq 1 $ for all $ y\in Y $. So $ f(p)=0 $ and $ f(q)=1 $ and $ 0\leq f(y) \leq 1 $ for all $ y\in Y $. Hence the result follows.
\end{proof}

\begin{theorem}\label{40}
If $ (Y,\gamma) $ is a $ s\lambda$-completely Hausdorff GTspace then following hold:

(1) for each pair of distinct points $ p,q $ of $ Y $,  there exist $ s\lambda $-open set $ E $ and $ sg_\lambda $-open set $ F $ such that $ p\in E, q\in F $ and $ q\in F\subset \overline{sF_\lambda}\subset Y-\overline{sE_\lambda} $;

(2) for each subset $ D $ and each singleton $ \{q\} $ of $ Y $ such that $ D\cap \{q\}=\emptyset $, there exist $ sg_\lambda $-open set $ E $ and $ s\lambda $-open set $ F $ such that $ D\cap E \not=\emptyset, \{q\} \subset F $ and $E\cap F=\emptyset$.
\end{theorem}

\begin{proof}
(1): Assume that $ (Y,\gamma) $ is a $ s\lambda$-completely Hausdorff GTspace and $ p,q $ are distinct points of $ Y $. Then, by Theorem \ref{39} (i), $ (Y,\gamma) $ is $ s\lambda $-Urysohn GTspace, so there exist $ s\lambda $-open sets $ E,F $ containing $ p,q $ respectively such that $ \overline{sE_\lambda}\cap \overline{sF_\lambda}=\emptyset $. Since $ F $ is also a $ sg_\lambda $-open set, then we have $ q\in F\subset \overline{sF_\lambda}\subset Y-\overline{sE_\lambda} $. 

(2): Suppose  $ (Y,\gamma) $ is a $ s\lambda$-completely Hausdorff GTspace, $ D\subset Y $ and $ q\in Y $ such that $ D\cap \{q\}=\emptyset $. Let $ p\in 
D $, then  $ p\not=q  $. Then, by Theorem \ref{39} (i), there exist $ s\lambda $-open sets $ E,F $ containing $ p,q $ respectively such that $ \overline{sE_\lambda}\cap \overline{sF_\lambda}=\emptyset $. So $ E\cap F=\emptyset $ and $ D\cap E\not=\emptyset, \{q\}\subset F $; $ E $ is also a $ sg_\lambda $-open set. 
\end{proof}

\begin{theorem}\label{41}
If $ (Y,\gamma) $ is $ s\lambda $-regular and $ s\lambda T_1 $ GTspace, then it is $ s\lambda $-Urysohn GTspace.
\end{theorem}
\begin{proof}
Let $ p,q\in Y, p\not=q $. For $ s\lambda T_1 $ GTspace, every singleton is $ s\lambda $-closed, by Lemma \ref{17}. So $ \{p\} $ is $ s\lambda $-closed and $ q\not\in \{p\} $, then there exist $ s\lambda $-open sets $ E,F$ such that $ \{p\}\in E, \{q\}\in F $ and $ \overline{sE_\lambda}\cap\overline{sF_\lambda}=\emptyset $. Hence the result follows.
\end{proof}

\begin{definition}\label{42}(c.f.\cite{SS}).
A GTspace $ (Y,\gamma) $ is said to satisfy $ s\lambda T_{3\frac{1}{2}} $ axiom if for each $ s\lambda $-closed set $ F $ and each point $ p\in Y, p\not\in F $, there always exists a $ s\lambda $-continuous real valued function $ f $ on $ Y $ such that $ f(p)=0 $ and $ f(y)=1 $ if $ y\in F $ and $ 0\leq f(y)\leq 1 $ for $ y\in Y $. 

The GTspace $ (Y,\gamma) $ which satisfies $ s\lambda T_{3\frac{1}{2}} $ axiom is called $ s\lambda $-completely regular GTspace. 
\end{definition}

\begin{definition}\label{43}(c.f.\cite{CG}).
A $ s\lambda $-completely regular $  s\lambda T_1$   GTspace $ (Y,\gamma) $ is called $ s\lambda T_{3\frac{1}{2}} $ GTspace or, is called $ s\lambda $-Tychonoff GTspace.
\end{definition}

\begin{definition}\label{43A}
(i): A $  s\lambda T_1 $   GTspace $ (Y,\gamma) $ is said to be $  s\lambda T_3 $ GTspace if it satisfies $  s\lambda T_3 $ axiom i.e. any $ s\lambda $-closed set $ P $ and a point $ q $ of $ Y, q\not\in P $, are $ s\lambda $-strongly separated in $ (Y,\gamma) $.

(ii): A $  s\lambda T_1 $   GTspace $ (Y,\gamma) $ is said to be $  s\lambda T_4 $ GTspace if it satisfies $  s\lambda T_4 $ axiom i.e. any two disjoint $ s\lambda $-closed sets of $ Y $ are $ s\lambda $-strongly separated in $ (Y,\gamma) $.

(iii): A $  s\lambda T_1 $   GTspace $ (Y,\gamma) $ is said to be $  s\lambda T_5 $ GTspace if it satisfies $  s\lambda T_5 $  axiom i.e. if two sets are $ s\lambda $-weakly separated, they are $ s\lambda $-strongly separated in $ (Y,\gamma) $.
\end{definition}

\begin{theorem}\label{44} Let $ (Y,\gamma) $ be a GTspace then  following hold:

(1) Every $ s\lambda T_4 $ GTspace satisfying condition (A) as mentioned in Lemma \ref{31} is  $ s\lambda $-Tychonoff.

(2) Every $ s\lambda $-Tychonoff GTspace $ (Y,\gamma) $ is $ s\lambda $-completely regular.

(3) Eevery $ s\lambda $-completely regular GTspace is $ s\lambda $-regular.

(4) Every $ s\lambda $-Tychonoff GTspace $ (Y,\gamma) $ is  $ s\lambda T_3 $.

(5) Every $ s\lambda $-Tychonoff GTspace $ (Y,\gamma) $ is $ s\lambda $-completely Hausdorff.
\end{theorem}

\begin{proof}
(1) Let $ (Y,\gamma) $ be a $  s\lambda T_4 $ GTspace satisfying the condition $ (A) $,  $ F $ be a $ s\lambda $-closed set and $ p\in Y $ such that  $p\not\in F $. As the GTspace is $ s\lambda T_1 $, $ \{p\} $ is $s\lambda$-closed. Since the GTspace is also $ s\lambda $-normal, then by Theorem \ref{36}, there exists a $ s\lambda $-continuous real valued function $ f $ on $ Y $ such that $ f(p)=0 $ and $ f(y)=1 $ if $ y\in F $ and $ 0\leq f(y)\leq 1 $ for all $ y\in Y $. Hence $ (Y,\gamma) $ is a $ s\lambda $-completely regular GTspace as well as a $ s\lambda $-Tychonoff GTspace.

(2) This follows immediately from the Definition \ref{43}.

(3) Let $ (Y,\gamma) $ be a $ s\lambda $-completely regular GTspace and $ F $ be a $ s\lambda $-closed set and $ p $ be a point of $ Y, p\not\in F $. Then $ \{p\} $ and $ F $ are separated by a $ s\lambda $-continuous real valued function. Hence by Theorem \ref{35}, $ \{p\} $ and $ F $ are separated by disjoint $ s\lambda $-closed neighbourhoods in  $ (Y,\gamma) $. Hence  $ (Y,\gamma) $ is $ s\lambda $-regular GTspace.

(4) It is an immediate consequence of the property (3) of this Theorem.

(5) Let $ (Y,\gamma) $ be a $ s\lambda $-Tychonoff GTspace and $ p,q\in Y, p\not=q $. Since $ (Y,\gamma) $ is $  s\lambda T_1$, the singleton $ \{p\} $ is $ s\lambda $-closed. By property (2) $ (Y,\gamma) $ is $ s\lambda $-completely regular. So $ p $ and $ q $ can be separated by a $ s\lambda $-continuous real valued function on $ Y $. Hence $ (Y,\gamma) $ is $ s\lambda $-completely Hausdorff GTspace.
\end{proof}

\begin{theorem}\label{45}
A $ s\lambda $-regular and $ s\lambda $-normal GTspace satisfying the condition (A) as mentioned in Lemma \ref{31} is $ s\lambda $-completely regular.
\end{theorem}
\begin{proof}
Let $ (Y,\gamma) $ be a $ s\lambda $-regular and $ s\lambda $-normal GTspace satisfying the condition $ (A) $. Let $ F $ be a $ s\lambda $-closed set and $ p\in X, p\not\in F $. Then $ p\in Y-F, $ a $ s\lambda $-open set. Then by Theorem \ref{15} (2), there exists a $ s\lambda $-open set $ V $ such that $ p\in V\subset \overline{sV_\lambda}\subset Y-F $. Since $ (X,\gamma) $ is also a $ s\lambda $-normal GTspace satisfying the condition (A), by Theorem \ref{36} the pair of disjoint $ s\lambda $-closed sets $ \overline{sV_\lambda} $ and $ F $ are separated by a $ s\lambda $-continous real valued function $ f $ on $ Y $. Thus $ f(p)=0 $ since $ p\in \overline{sV_\lambda}, f(y)=1 $ for all $ y\in F $ and $ 0\leq f(y)\leq 1 $ for all $ y\in Y $. Hence $ (Y,\gamma) $ is $ s\lambda $-completely regular GTspace.
\end{proof}

\begin{theorem}\label{46}
If $ (Y,\gamma) $ is $ s\lambda$-completely regular GTspace, then the condition (1) holds:

(1) for each $ s\lambda $-closed singleton $ \{y\}\subset
Y $ and for each $sg_\lambda $-open set $ G $ containing $ y $, there exists a $ sg_\lambda $-open set $ E $ such that $ y\in E\subset \overline{sE_\lambda}\subset sInt_\lambda(G) $.

(2) If for a GTspace $ (Y,\gamma) $ the condition (1) holds then for each $ s\lambda $-closed singleton $ \{y\}\subset
Y $ and for each $s\lambda $-open set $ G, y\in G $, there exists a $ sg_\lambda $-open set $ E $ such that $ y\in sInt_\lambda(E)\subset E\subset \overline{sE_\lambda}\subset G $.

(3) For each $ sg_\lambda $-closed singleton $ \{y\}\subset
Y $ and for each $s\lambda $-open set $ G $ containing $ y $, there exists a $ sg_\lambda $-open set $ E $ such that $ \overline{s\{y\}_\lambda}\subset E\subset \overline{sE_\lambda}\subset G $ if $ (Y,\gamma) $ is $ s\lambda $-normal.

(4) For each $ sg_\lambda $-closed singleton $ \{y\}\subset
Y $ and for each $s\lambda $-open set $ G $ containing $ y $, there exists a $ s\lambda $-open set $ E $ such that $ \overline{s\{y\}_\lambda}\subset E\subset \overline{sE_\lambda}\subset G $, if the GTspace is $ s\lambda $-normal.
\end{theorem}
\begin{proof}
(1)  Let $ \{y\} $ be a $ s\lambda $-closed singleton and $ G $ be a  $ sg_\lambda $-open set of a $ s\lambda $-completely regular GTspace $ (Y,\gamma) $ and  $ \{y\}\subset G $. So by Theorem \ref{9}, $ \{y\}\subset sInt_\lambda(G) $, a $ s\lambda $-open set. By Theorem \ref{44} (3),  $ (Y,\gamma) $ is $ s\lambda $-regular, then by Theorem \ref{15} (2), there exists a $ s\lambda $-open set $ E $ which is also a $ sg_\lambda $-open set such that $ y\in
E\subset \overline{sE_\lambda}\subset sInt_\lambda(G) $. 

(2)  Suppose $ \{y\} $ is a $ s\lambda $-closed singleton and $ G $ is a  $ s\lambda $-open set of $ Y $ and $ \{y\}\subset G $. So $ G $ is $ sg_\lambda $-open. Then by (1), there exists a $ sg_\lambda $-open set $ E $  such that $ y\in
E\subset \overline{sE_\lambda}\subset sInt_\lambda(G) $. Since $ E $ is $ sg_\lambda $-open containing $ s\lambda $-closed set $ \{y\} $ so $ \{y\}\subset sInt_\lambda(E) $. So $ y\in sInt_\lambda(E)\subset E\subset \overline{sE_\lambda}\subset G $.

(3) Suppose $ (Y,\gamma) $ is a $ s\lambda $-normal GTspace, $ \{y\} $ is a $ sg_\lambda $-closed set and $ G $ is a  $ s\lambda $-open set of $ Y $ and $ \{y\} \subset G $, then $ G \supset \overline{s\{y\}_\lambda} $,  a $ s\lambda $-closed set. By Theorem \ref{30} (2), there exists a $ sg_\lambda $-open set $ E $ such that $ \overline{s\{y\}_\lambda}\subset E\subset\overline{sE_\lambda}\subset  G $. 

(4) Let $ (Y,\gamma) $ be a $ s\lambda $-normal GTspace and $ \{y\} $ be a $ sg_\lambda $-closed singleton contained in a  $ s\lambda $-open set   $ G $ of $ Y $, then $ G\supset \overline{s\{y\}_\lambda} $,    a $ s\lambda $-closed set. By Theorem \ref{30} (2), there exists a $ sg_\lambda $-open set $ E $ such that $ \overline{s\{y\}_\lambda}\subset E\subset\overline{sE_\lambda}\subset  G $. Since $ E $ is $ sg_\lambda $-open and $\overline{s\{y\}_\lambda} $ is $ s\lambda $-closed, then  by Theorem \ref{9}, $\overline{s\{y\}_\lambda}\subset sInt_\lambda(E) $. Put $ F= sInt_\lambda(E)$, then $ F $ is $ s\lambda $-open and hence $ \overline{s\{y\}_\lambda}\subset F\subset \overline{sF_\lambda}=\overline{s(sInt_\lambda(E))_\lambda}\subset \overline{sE_\lambda}\subset G $.
\end{proof}

\begin{definition}\label{48}
A GTspace $ (Y,\gamma) $ is said to be $ s\lambda $-completely normal if for each pair of  $ s\lambda $-weakly separared sets $ G, H $ (as mentioned in Note \ref{7A}), there exist $ s\lambda $-open sets $ U, V $ such that $ G\subset U, H\subset V $ and $ \overline{sU_\lambda}\cap \overline{sV_\lambda}=\emptyset $. 
\end{definition}

\begin{theorem}\label{49}
Every subspace of a $ s\lambda $-completely normal GTspace is $ s\lambda $-normal. 
\end{theorem}

\begin{proof}
Let $ (Y,\gamma) $ be a $ s\lambda $-completely normal GTspace and $ X $ be a subset of it. Let $ F, K $ be a pair of disjoint $ s\lambda $-closed sets in $ (X, \gamma_X), \gamma_X $ being the subspace generalized topology in $ X $. So there exist two $ s\lambda $-closed sets $ F_1, K_1 $ in $ (Y,\gamma) $ such that $ F=F_1\cap X, K=K_1\cap X $. Since $ F\cap K=\emptyset$ and $ (F_1-F)\cap K=\emptyset $ (as $ F_1-F\subset Y-X $ and $ F\subset X), F_1\cap K=\{(F_1-F)\cup F\}\bigcap K=\{(F_1-F)\cap K\}\cup \{F\cap K\}=\emptyset $. Similarly, $ F\cap K_1=\emptyset $. Consequently, $ F, K $ are $ s\lambda $-weakly separated in $ (Y,\gamma) $ by two $ s\lambda $-open sets $ Y-K_1 $ and $ Y-F_1 $. As $ (Y,\gamma) $ is $ s\lambda $-completely normal, there exist $ s\lambda $-open sets $ U, V $ in $ (Y,\gamma) $ such that $ F\subset U, K\subset V $ and $ \overline{sU_\lambda}\cap \overline{sV_\lambda}=\emptyset $. Then $ F\subset (U\cap X)\subset U,  K\subset (V\cap X)\subset V $ where $ (U\cap X),  ( V\cap X) $ are $ s\lambda $-open sets in $ (X, \gamma_X) $. Now $ s\lambda $-closures of $ (U\cap X) $ and $ (V\cap X) $ in $ (X, \gamma_X) $ are respectively  $ X\bigcap \overline{s(U\cap X)_\lambda} $ and $ X\bigcap \overline{s(V\cap X)_\lambda} $, therefore, $ X\bigcap \overline{s(U\cap X)_\lambda} \bigcap X\bigcap \overline{s(V\cap X)_\lambda}=X\bigcap (\overline{s(U\cap X)_\lambda}\cap \overline{s(V\cap X)_\lambda})\subset X\bigcap (\overline{sU_\lambda}\cap \overline{sV_\lambda})=X\cap \emptyset=\emptyset $. Hence $ (X, \gamma_X) $ is $ s\lambda $-normal. \end{proof}

\begin{theorem}\label{49A} Let $ (Y,\gamma) $ be a GTspace    satisfying the condition (A) as mentioned in Lemma \ref{31}.
If every subspace of it  is $ s\lambda $-normal, then every pair of $ s\lambda $-weakly separated set is $ s\lambda $-strongly separated (as mentioned in Note \ref{7A}). 
\end{theorem}

\begin{proof}
Let the condition hold and $ G,H $ be two $ s\lambda $-weakly separated sets of $ Y $. Then by Lemma \ref{8A},  $ G\cap \overline{sH_\lambda}=\emptyset= H\cap \overline{sG_\lambda} $. Assume $ E=Y-( \overline{sG_\lambda} \cap \overline{sH_\lambda})\Rightarrow E=(Y- \overline{sG_\lambda}) \cup (Y-\overline{sH_\lambda})\Rightarrow E $ is a $ \lambda $-open set in $ (Y,\gamma) $. Now $ E\cap \overline{sG_\lambda} $ and $ E\cap \overline{sH_\lambda} $ are two disjoint $ s\lambda $-closed sets of $ (E,\gamma_E)$, a $ s\lambda $-normal GTspace  by assumption, where  $ \gamma_E $ is subspace generalized topology. Then there exist $ s\lambda $-open sets $ U,V $ in $ (E,\gamma_E) $ containing $ E\cap \overline{sG_\lambda} $ and $ E\cap \overline{sH_\lambda} $ respectively such that $ \overline{sU_\lambda} \cap \overline{sV_\lambda}=\emptyset $,  the   $ s\lambda $-closures of $ U, V $ being taken with respect to $ (E,\gamma_E) $. Since the GTspace satisfies the condition (A) and $ E $ is a $ \lambda $-open set in $ (Y,\gamma) $; $ U,V $ are $ s\lambda $-open in $ (Y,\gamma) $. Also  $ E\cap \overline{sG_\lambda}=[(Y-\overline{sG_\lambda})\cup(Y-\overline{sH_\lambda})]\bigcap \overline{sG_\lambda}=\{(Y-\overline{sG_\lambda})\cap \overline{sG_\lambda}\}\bigcup\{(Y-\overline{sH_\lambda})\cap\overline{sG_\lambda}\}=(Y-\overline{sH_\lambda})\bigcap \overline{sG_\lambda}\supset G $, since $ G\subset Y-\overline{sH_\lambda} $. Similarly, $ E\cap \overline{sH_\lambda}\supset H $. Thus $ G\subset E\cap \overline{sG_\lambda}\subset U, H\subset E\cap \overline{sH_\lambda}\subset V $ and $ U\cap V=\emptyset $. Hence the result follows.  
\end{proof}

\begin{corollary}\label{49B}
Let $ (Y,\gamma) $ be a GTspace    satisfying the condition (A) as mentioned in Lemma \ref{31}.
If every subspace of it  is $ s\lambda $-closed and $ s\lambda $-normal, then it is $ s\lambda $-completely normal. 
\end{corollary}
\begin{proof}
We consider the set $ E $ as in the proof of the Theorem \ref{49A}, so $ (E, \gamma_E) $ is a subspace. Let $ G $ and $ H $ be two $ s\lambda $-weakly separated sets in a GTspace $ (Y,\gamma) $. Then  by Theorem \ref{49A}, we have a pair of $ s\lambda $-open sets $ U $ and $ V $ which separate the sets $ G $ and $ H $ $ s\lambda $-strongly.  Now $ s\lambda $-closure of $ U,  V$ in $ (E, \gamma_E) $ are $ E\cap \overline{sU_\lambda} $ and $ E\cap \overline{sV_\lambda} $ where $  \overline{sU_\lambda} $ and $  \overline{sV_\lambda} $ are the $ s\lambda $-closure of  $ U, V $ in $ (Y,\gamma) $. As the subspace $ E $ is $ s\lambda $-closed by supposition, $ s\lambda $-closure of $ U, V $ in $ (E, \gamma_E) $ are the same as the $ s\lambda $-closure of $  U, V $ in $ (Y,\gamma) $. Since $  U\subset  E, \overline{sU_\lambda} \subset \overline{sE_\lambda}=E\Rightarrow  E\cap \overline{sU_\lambda} = \overline{sU_\lambda}  $. Similarly $  E\cap \overline{sV_\lambda} = \overline{sV_\lambda}  $. Hence the result follows.
\end{proof}

\begin{definition}\label{49C}(c.f.\cite{CG}).
A GTspace $ (Y ,\gamma) $ is said to be a $ s\lambda $-perfectly normal if for each pair of disjoint $ s\lambda $-closed sets $ M, N $ there exists a $ s\lambda $-continuous real valued function $ f $ on $ Y $ such that $ 0\leqslant f(y)\leqslant 1 $  for all $ y\in Y $ and $ f^{-1}(0)=M $ and $ f^{-1}(1)=N $.
\end{definition}

\begin{remark}\label{50}
Every $ s\lambda $-completely normal GTspace  is $ s\lambda $-normal and every $ s\lambda T_5 $ GTspace is $ s\lambda T_4 $ (as defined in Definition \ref{43A}).  Obviously, if a GTspace is $ s\lambda T_5 $  every subspace of it is $ s\lambda T_4 $ and  every $ s\lambda $-perfectly normal GTspace is $ s\lambda $-normal. \end{remark}

\begin{lemma}\label{51}
If $ f $ is a $ s\lambda $-continuous real valued function on a GTspace $ (Y, \gamma) $, the set $ B=\{y\in Y: f(y)< p\} $ is $ s\lambda $-open. 
\end{lemma}
\begin{proof}
As the interval $ (-\infty, p) $ is open in real number space $ (R,\sigma) $ where $ \sigma $ is usual topology on $ R $, then the set $ B $ is $ s\lambda $-open, by Definition \ref{33}. 
\end{proof}

\begin{lemma}\label{52} (c.f.\cite{CG}).
Let $ (Y,\gamma) $ be a $ s\lambda $-normal GTspace satisfying the condition (A) as mentioned in Lemma \ref{31}. Then a necessary and sufficient condition for the existence of a $ s\lambda$-continuous real valued function $ f $ on $ Y $ for a pair of disjoint $ s\lambda $-closed sets $ M,  N $  satisfying $ f^{-1}(0)=M $ is that $ M $ is a $ s\lambda G_\delta $-set (as mentioned in Note \ref{7A}).
\end{lemma}

\begin{proof}
Suppose there exists a  $ s\lambda$-continuous real valued function $ f $ on $ Y $ for a pair of disjoint $ s\lambda$-closed sets $ M,N $ which satisfies $ f^{-1}(0)=M $ and $ U_n=\{y\in Y: f(y)<\frac{1}{n}, n=1,2,.....  \} $. Then by Lemma \ref{51}, each $ U_n $ is a $ s\lambda $-open set. Hence $ M=\bigcap\{U_n:n=1,2,......\} $ is a $ s\lambda G_\delta $-set.

Conversely, let $ (Y,\gamma) $ be a $ s\lambda $-normal GTspace satisfying the condition (A) as mentioned in Lemma \ref{31}. Suppose $ M,N $ are disjoint $ s\lambda $-closed sets and $ M $ is a $ s\lambda G_\delta $-set i.e. $ M=\bigcap\{U_n:n=1,2,......\} $ say, where each $ U_n $ is a $ s\lambda $-open set. Let $ V_n=U_1\cap U_2\cap........\cap U_n $, then by condition (A), $ V_n $ is also a $ s\lambda $-open set for $ n=1,2,...... $ and $ M=\bigcap \{V_n:n=1,2,........\} $. Also $ V_1\supset V_2\supset ........ . $ and by $ s\lambda $-normality criterion, we can assume that $ V_1\cap N=\emptyset $. Corresponding to  each pair of disjoint $ s\lambda $-closed sets $ \{M, Y-V_n\} $ there exists, by Theorem \ref{37}, a $ s\lambda $-continuous real valued function $ f_n $ on $ Y $ for each $ n $  and $ f_n(M)=0, f(Y-V_n)=1 $ and $ 0\leqslant f_n(y)\leqslant 1 $ for all $ y\in Y $. Let $ f=\Sigma(\frac{1}{2^n})f_n $. Then $ \frac{f_n(y)}{2^n}\leqslant \frac{1}{2^n}=K_n $ say, for $ n=1,2,....... $ and for all $ y\in Y $. So $ K_n $ are positive constants and hence the series $ \Sigma K_n $ is convergent. Therefore, the infinite series $ \Sigma(\frac{1}{2^n})f_n $ is uniformly convergent (Weierstrass' M-test) and converges to $ f $ on $ Y $ and $ f $ is also a $ s\lambda $-continuous real valued function on $ Y $ and $ f(M)=0, f(N)=1 $ and $ 0\leqslant f(y)\leqslant 1 $ on $ Y $. Now we claim that there is no point $ y $ outside $ M $ where $ f(y)=0 $. Let $ y\not\in M $, then $ y\not\in V_r\Rightarrow y\in Y-V_r $ for some $ r $ and $ N\subset Y-V_r $. So $ f_r(y)\not=0 $ but $ \leqslant 1 $ and therefore, $ f(y)\geqslant \frac{1}{2^r} $ i.e. $ f(y)\not=0 $; consequently, $ f^{-1}(0)=M $.
\end{proof}

\begin{theorem}\label{53}
Let $ (Y,\gamma) $ be a $ s\lambda $-perfectly normal GTspace satisfying the condition (A) as mentioned in Lemma \ref{31} such that every $ s\lambda $-closed set in it is a $ s\lambda G_\delta $-set; then every pair of $ s\lambda $-weakly separated sets is $ s\lambda $-strongly separated.
\end{theorem}

\begin{proof}
Let the condition hold and $ G, H $ be $ s\lambda $-weakly separated sets in a $ s\lambda $-perfectly normal GTspace $ (Y, \gamma) $.   Then by Lemma \ref{8A}, $ G\cap \overline{sH_\lambda}=\emptyset = H\cap \overline{sG_\lambda} $. Hence, for any $ s\lambda $-closed set $ P $  with $  \overline{sG_\lambda}\cap P=\emptyset $, there exists, by Lemma \ref{52}, a $ s\lambda $-continuous real valued function $ f $ on $ Y $ corresponding to the pair $\{\overline{sG_\lambda}, P\} $ satisfying the condition $ f^{-1}(0)=\overline{sG_\lambda} $. Similarly, for any $ s\lambda $-closed set $ Q $  with $ \overline{sH_\lambda}\cap Q=\emptyset $, there exists a $ s\lambda $-continuous real valued function $ g $ on $ Y $ corresponding to the pair $\{\overline{sH_\lambda}, Q\} $ satisfying the condition $ g^{-1}(0)=\overline{sH_\lambda} $. Let $ U=\{y\in Y: f(y) < g(y)\} $ and $ V=\{y\in Y: g(y) < f(y)\} $. Then $ U, V $ are $ \gamma $-open sets, so $  s\lambda $-open sets
and $ U\cap V=\emptyset $. Now we assert that  $ G\subset U $. For, if $ y\in G $, then $ f(y)=0 $ and $ \overline{sH_\lambda}\cap G=\emptyset $ and so we have $ g(y)>0=f(y) $. Similarly, $ H\subset V $. Hence the result follows.
\end{proof}

\begin{theorem}\label{54}
Let $ (Y,\gamma) $ be a GTspace, then the following hold:

(i) $ (Y,\gamma) $ is a $ s\lambda $-completely normal GTspace;

(ii) for each pair of $ s\lambda $-weakly separated sets $ A,B $ of $ Y, $ there exist $ sg_\lambda $-open sets $ U, V $ and $ s\lambda $-open sets $ E, F $ such that 

$ A\subset U\subset \overline{sU_\lambda}\subset E \subset \overline{sE_\lambda}\subset Y-\overline{sV_\lambda} $  and $ B\subset V\subset \overline{sV_\lambda}\subset F \subset \overline{sF_\lambda}\subset Y-\overline{sU_\lambda} $;

(iii) for each pair of sets $ A, B $ of $ Y $ satisfies $ A\cap Q=\emptyset=B\cap P $ where $ P,Q $ are $ s\lambda $-closed sets, there exist $ sg_\lambda $-open sets $ U, V $  and $ s\lambda $-open sets $ E, F $ such that 

$ A\subset U\subset \overline{sU_\lambda}\subset E \subset \overline{sE_\lambda}\subset Y-\overline{sV_\lambda} $  and $ B\subset V\subset \overline{sV_\lambda}\subset F \subset \overline{sF_\lambda}\subset Y-\overline{sU_\lambda} $;

(iv) for each $ s\lambda $-closed set $ P $ and each $ s\lambda $-open set $ G, G\supset P $, there exists a $ sg_\lambda $-open set $ E $ such that $ P\subset E\subset \overline{sE_\lambda} \subset G $. 
\end{theorem}

\begin{proof} 
$ (i)  \Rightarrow (ii) $: Assume $ A,B $ are two $ s\lambda $-weakly separated sets of a $ s\lambda $-completely normal GTspace $ (Y,\gamma) $. So there exist $ s\lambda $-open sets $ U,V $ which are also $ sg_\lambda $-open sets containing $ A,B $ respectively such that $ \overline{sU_\lambda}\cap \overline{sV_\lambda}=\emptyset \Rightarrow \overline{sU_\lambda}\subset Y-\overline{sV_\lambda} $, a $ s\lambda $-open set. Again $ (Y,\gamma) $ is also a $ s\lambda $-normal, so by Theorem \ref{29} (3), there exist $ s\lambda $-open set $ E $ such taht $  \overline{sU_\lambda}\subset E\subset \overline{sE_\lambda}\subset Y-\overline{sV_\lambda} $. So  $ A\subset U\subset \overline{sU_\lambda}\subset E\subset \overline{sE_\lambda}\subset Y-\overline{sV_\lambda} $ Similarly, $\overline{sV_\lambda}\subset Y-\overline{sU_\lambda} $, a $ s\lambda $-open set and hence there exist $ s\lambda $-open set $ F $ such that $  \overline{sV_\lambda}\subset F\subset \overline{sF_\lambda}\subset Y-\overline{sU_\lambda} $ and hence $ B\subset V\subset \overline{sV_\lambda}\subset F\subset \overline{sF_\lambda}\subset Y-\overline{sU_\lambda} $.
 
$ (ii) \Rightarrow (iii) $: From the assumption we have $ A\subset Y-Q, B\subset Y-P $ where $ (Y-Q), (Y-P)$ are $ s\lambda $-open sets and $ A\cap (Y-P)=\emptyset=B\cap(Y-Q) $. So $ A,B $ are $ s\lambda $-weakly separated. Then by (ii), the result follows.
 
$ (ii) \Rightarrow(iv) $: Here $ P\subset G $, so $ (Y-G)\subset Y-P $ where $ G, (Y-P) $ are $ s\lambda $-open sets and $ P\cap (Y-P)=\emptyset=G\cap (Y-G) $. Then, $ P, Y-G $ are $ s\lambda $-weakly separated sets. So, by (ii), there exist $ sg_\lambda $-open sets $ U,V $ and $ s\lambda $-open sets as well as $ sg_\lambda $-open sets $ E,F $ such that $ P\subset U\subset \overline{sU_\lambda}\subset E\subset \overline{sE_\lambda}\subset Y-\overline{sV_\lambda} $ and $ (Y-G)\subset V\subset \overline{sV_\lambda}\subset F\subset \overline{sF_\lambda}\subset Y-\overline{sU_\lambda} $. Hence $ (Y-G)\subset \overline{sV_\lambda}\Rightarrow (Y-\overline{sV_\lambda})\subset G $. This implies that $ P\subset E \subset \overline{sE_\lambda} \subset G $. 
\end{proof}


\begin{thebibliography}{99}\baselineskip=16pt

\bibitem{AD}   Alexandroff, A. D., \textit{Additive set functions in abstract space,} Mat. Sb. (N.S.) \textbf{8} (50) (1940), 307-348 (English, Russian Summary).

\bibitem{AA}   Al-Hawary, Talal Ali and Al-Omari, Ahmed, \textit{Quasi b-open sets in bitopological spaces,} Abhath Al-Yarmouk: ``Basic Sci. and Eng."  \textbf{Vol 21}, No-1 (2012), pp 1-14.

\bibitem{BN} Bachman, G. and  Narici, L., \textit{Functional Analysis,} \textit{Academic Press, New York and London}. 



 

\bibitem{AR}    Banerjee, A. K. and Mondal, R., \textit{A note on connectedness in a bispace}, \textit{Malaya. J. Mat.}, \textbf{5} (1), 104-108.

\bibitem{JP1} Banerjee, A. K.  and Pal, J.,  \textit  { $\lambda^*$-closed sets and new separation axioms in Alexandroff spaces}, \textit{Korean J. Math.,}   \textbf{26} (2018). No. 4, 709-727. 




\bibitem{JP2} Banerjee, A. K.  and Pal, J.,  \textit  {Semi $\lambda^*$-closed sets and new separation axioms in Alexandroff spaces}, \textit{South East Asian J. of Math.$\&$ Math. Sci.,}   \textbf{14} (1) (2018), 115-134. 


\bibitem{BP}  Banerjee, A. K. and Saha, P. K., \textit{Pairwise semi bicompact and pairwise semi Lindeloff  bispaces}, \textit{Int. J. Math. Sci. and Engg. Appls.}, \textbf{11} (II),  47-57.

\bibitem{BL} Bhattacharyya, P. and  Lahiri, B. K., \textit{Semi-generalised closed sets in topology,} \textit{Indian J Math.} \textbf{29} (1987), 376-382. 

\bibitem{BC} Boonpok, C.,  \textit{Generalized $(\Lambda, b)$-closed sets in topological spaces},  \textit{Korean J. Math.,}   \textbf{25} (2017), No. 3. pp 437-453. 

\bibitem{AC} A. Cs$\acute{a}$sz$ \acute{a} $r,  Generalized topology,   generalized continuity,  \textit{Acta Math. Hungar.,}   \textbf{96} (2002), 351-357.

\bibitem{CA} A. Cs$\acute{a}$sz$ \acute{a} $r,  Generalized open sets,  \textit{Acta Math. Hungar.,}   \textbf{75} (1997), 65-87.


\bibitem{CG}  Chatterjee, B. C., Ganguly,  Adhikary,  \textit{ A  Text Book of Topology ,}\textit{ Asian Books Private Limited, New Delhi-110 002}.






  

\bibitem{DM} Dontchev, J. and Maki, H.,  \textit{On $ sg$-closed  sets and semi-$\lambda$-closed sets  }, \textit{Questions Ans. Gen. Topology}, \textbf{15 (2)} (1997), 259-266.

\bibitem{LD} Lahiri, B. K. and  Das, P.,  \textit{Certain bitopological concepts in a bispace},  \textit{Soochow Journal of Mathematics}, \textbf{27 (2)} (2001), 175-185.




\bibitem{NL} Levine,  N.,  \textit{Generalised closed sets in topology}, Rend. Cire.  Mat. Palermo \textbf{19} (2) (1970), 89-96.

\bibitem{MBD} Maki, H., Balachandran, K. and Devi, R.,  \textit{Remarks on semi-generalised closed  sets and generalised semi closed sets},  \textit{Kyungpook Math. J.}, \textbf{36 (1)} (1996), 155-163.


\bibitem{JRM} Munkres, J. R.,   \textit{Topology, Second Edition},  \textit{Massachusetts Institute of Technology}, \textbf{Prentice-Hall of I}.

\bibitem{AJ}  Pal. J. and Banerjee, A. K., \textit{$ s\beta_\lambda $-closed sets and some low separation axioms in GT-spaces}, \textit{Archived}. 

\bibitem{MS} Sarsak,  M. S., \textit{New separation axioms in generalized topological spaces}, \textit{Acta Math. Hungar.,} \textbf{132 (3)} (2011), 244-252.


\bibitem{SM} Sarsak, M. S.,  \textit{Weak separation axioms in generalized topological spaces}, \textit{Acta Math. Hungar.,} \textbf{131} (1-2) (2011), 110-121.

\bibitem{SS} Steen,  L. A. and Seebach, J. A., Jr., \textit{Counterexamples in topology}, \textit{Second Edition,} \textit{Springer-Verlag, New York, Berlin, Heidelberg, Tokyo}.



  
\end{thebibliography}
\end{document}